\documentclass[11pt]{article}
\usepackage[margin=1.25in]{geometry}

\usepackage{amsfonts}
\usepackage{color}

\usepackage{amsmath}
\usepackage{graphicx}
\usepackage{multirow}
\usepackage{color}
\usepackage[square,sort,comma,numbers]{natbib}
\RequirePackage[colorlinks,citecolor=blue,urlcolor=blue]{hyperref}
\usepackage{amssymb}
\usepackage[british]{babel}
\usepackage{subfigure}

\usepackage[utf8]{inputenc}

\usepackage{amsthm}
\usepackage{latexsym}
\usepackage{mathrsfs}
\usepackage[mathscr]{euscript}
\usepackage{array}

\newtheorem{theorem}{Theorem}

\newtheorem{definition}[theorem]{Definition}
\newtheorem{example}[theorem]{Example}

\newtheorem{lemma}[theorem]{Lemma}

\newtheorem{proposition}[theorem]{Proposition}

\DeclareMathOperator*{\esssup}{ess\,sup}

    \newcommand\ds{\displaystyle}

	\newcommand{\grad}{\operatorname{grad}}

    \DeclareMathOperator*{\essmax}{ess\,max}
    \DeclareMathOperator*{\argessmax}{arg\,ess\,max}
	\renewcommand{\theta}{\vartheta}
	\let\temp\epsilon
	\let\epsilon\varepsilon
	\let\varepsilon\temp
\definecolor{myblue}{rgb}{0.0,0.2,0.5}
	\definecolor{darkblue}{rgb}{0,0,0.5}

\title{Capital allocation for cash-subadditive risk measures: from BSDEs to BSVIEs}
\author{Emanuela Rosazza Gianin\footnote{Department of Statistics and Quantitative Methods,
University of Milano Bicocca,
Via Bicocca degli Arcimboldi 8, 20126 Milano, Italy.
emanuela.rosazza1@unimib.it} \, and Marco Zullino \footnote{Department of Mathematics and Applications,
University of Milano Bicocca, 20126 Milano, Italy.
m.zullino@campus.unimib.it}}
\providecommand{\keywords}[1]
{
  \small	
  \textbf{\textit{Keywords:}} #1
}
\begin{document}

\maketitle

\begin{abstract}
In the context of risk measures, the capital allocation problem is widely studied in the literature where different approaches have been developed, also in connection with cooperative game theory and systemic risk. Although static capital allocation rules have been extensively studied in the recent years, only few works deal with dynamic capital allocations and its relation with BSDEs. Moreover, all those works only examine the case of an underneath risk measure satisfying cash-additivity and, moreover, a large part of them focuses on the specific case of the gradient allocation where Gateaux differentiability is assumed.

The main goal of this paper is, instead, to study \textit{general dynamic capital allocations} associated to \textit{cash-subadditive} risk measures, generalizing the approaches already existing in the literature and motivated by the presence of (ambiguity on) interest rates. Starting from an axiomatic approach, we then focus on the case where the underlying risk measures are induced by BSDEs whose drivers depend also on the $y$-variable. In this setting, we surprisingly find that the corresponding capital allocation rules solve special kinds of Backward Stochastic Volterra Integral Equations (BSVIEs).
\end{abstract}

\keywords{risk measures; capital allocation; BSDE; BSVIE; cash-subadditivity; subdifferential}

\section{Introduction}

Starting from the seminal paper of Artzner et al. \cite{ArDeEbHe99}, a wide literature on Mathematical Finance has been devoted to the theory of risk measures that have been introduced and studied - both from an axiomatic and an empirical point of view- in order to quantify the riskiness of financial exposures in a static and in a dynamic setting. See \cite{ArDeEbHe99,bion1,delb,drapeau-kupper,foellmer,puttingorder,Bion,CerreiaVioglio,Frittelli2,det,ElRa}, among many others, for an axiomatic treatment of the topic. In the context of dynamic risk measures, the relation between risk measures and backward equations have been deeply analysed: firstly, for Backward Stochastic Differential Equations (BSDEs) in \cite{BaEl,penalty,DN-RG,ElRa,viag}; and, recently, for Backward Stochastic Volterra Integral Equations (BSVIEs) in \cite{Nacira,DN-RG,Yong2007}. We recall that, while BSDEs date back to  \cite{pardoux,peng-97} and their applications to Mathematical Finance are well known in the literature (see \cite{EPQ}), BSVIEs were introduced for the first time by \cite{Yong2006} and recently investigated also in view of their applications.

During the last years and in the context of risk measures, an increasing number of studies has been focused on the capital allocation problem that, roughly speaking, consists of finding a ``fair'' way (i.e. satisfying some suitable criteria) to divide into the different components of the risk the margin to be deposited because of the risk exposure. To be more concrete, with different components we can think at different sub-units which a risky position is formed of (e.g. a portfolio composed by different assets or a firm formed by different business lines or branches).
On the capital allocation problem, several papers deal with an axiomatic approach in a static framework and with its relation with cooperative game theory (see \cite{centrone,delb,delnaut,Kalkbrener,Tasche,tsanakas} and the references therein). Some recent works, instead, focus also on capital allocation rules (CARs) in the dynamic case, even if several of those restrict their attention mainly on the gradient allocation (see, e.g.,  \cite{boonen,KO,KO1,tsanakas1}). Furthermore, quite recently, the strong relation between risk measures and BSDEs led to a natural extension of the concept of a capital allocation rule to a dynamic environment by means of BSDEs. In particular, \cite{KO} has proved that the dynamic gradient capital allocation of a dynamic risk measure induced by a BSDE follows a BSDE. Moreover, \cite{MastroRos} gives an organic and axiomatic treatment of \textit{general} dynamic CARs as well as a general construction of dynamic CARs by means of BSDEs.

It is worth emphasizing that all the aforementioned papers deal with capital allocation rules associated to (coherent or convex) cash-additive risk measures. As argued by \cite{ElRa}, however, in the presence of stochastic interest rate and/or ambiguity on the interest rate the axiom of cash-additivity should be relaxed, e.g., with the so-called cash-subadditivity. See also \cite{CerreiaVioglio,drapeau-kupper,Frittelli2,MastroRos1,MastroRos2} for the impact of non cash-additivity on risk measures and for a deeper discussion on cash-subadditivity versus cash-additivity.

Motivated by the importance of studying general dynamic CARs and by the financial need of weakening cash-subadditivity, our aim is here to develop an \textit{axiomatic treatment of CARs based on dynamic risk measures that satisfy cash-subadditivity} but not necessarily Gateaux differentiability. In order to pursue this objective, we generalize the approach followed in \cite{MastroRos}, covering the case of cash-subadditivity. Differently from \cite{MastroRos1,KO1,KO} where the authors proved that for cash-additive risk measures induced by BSDEs (or BSVIEs) the corresponding capital allocations follow still a BSDE (or a BSVIE), in this paper we prove that the subdifferential CAR associated to a cash-subadditive risk measure induced via a BSDE follows a BSVIE. More in general, we introduce a \textit{general technique to build capital allocations following BSVIEs}, similarly as done for cash-additive risk measures in \cite{MastroRos,MastroRos1},
and we provide a link between dynamic cash-subadditive risk measures following a BSDEs and capital allocations based on BSVIEs. This new approach will be shown to cover relevant examples of capital allocation principles, such as the gradient, the Aumann-Shapley, and the marginal ones.  \medskip

The paper is organized as follows. In Section \ref{sec: prelim} we give a short review on dynamic risk measures, capital allocation rules, BSDEs, and BSVIEs. Section \ref{sec: CAR and csa} contains our main results. Firstly, we provide a one-to-one correspondence between dynamic cash-subadditive convex risk measures and capital allocations rules satisfying some further properties, by generalizing a similar result proved in \cite{MastroRos} for the cash-additive case. Furthermore, focusing
   on CARs whose underlying cash-subadditive risk measures are driven by BSDEs,
 we show that the subdifferential CAR follows a BSVIE. More in general, we provide a natural way to build different CARs whose dynamics follow a BSVIE with a driver that is related to that of the associated risk measure. Section \ref{sec: examples} contains some examples based on marginal capital allocations and on the entropic risk measure. Finally, in Section \ref{sec: conclus} some conclusions and final remarks are provided.

\section{Short review and preliminary results} \label{sec: prelim}
In this section we fix some notations and we recall the notions that will be used in the following. We start fixing some relevant notations used throughout the paper, then we recall fundamental properties of dynamic risk measures. Furthermore, we provide an insight into the main results regarding BSDEs and BSVIEs, proving some new and useful results for our specific purposes.\medskip


Let $T>0$ be a future time horizon and consider a filtered probability space $(\Omega,\mathcal{F},(\mathcal{F}_t)_{t\in[0,T]},\mathbb{P})$ with $\mathcal{F}_0=\{\emptyset,\Omega\}$ and $\mathcal{F}_T=\mathcal{F}$. In the sequel, any equality and inequality has to be understood $\mathbb{P}$-almost surely (shortly, a.s.), unless otherwise specified.

Let $L^p(\mathcal{F}_t)$, with $p \in [1,+\infty)$, denote the space formed by all ($\mathbb{R}^n$-valued) random variables on $(\Omega, \mathcal{F},\mathbb{P})$ that are $\mathcal{F}_t$-measurable and such that $\mathbb{E}\left[|X|^p\right]<+\infty$, while  $L^{\infty}(\mathcal{F}_t)$ the space formed by all ($\mathbb{R}^n$-valued) random variables that are $\mathcal{F}_t$-measurable and essentially bounded. Furthermore, $L^p_+(\mathcal{F}_t)$, with $p \in [1,+\infty]$, will denote the subset of $L^p(\mathcal{F}_t)$ of random variables that are, $\mathbb{P}$-a.s., greater than or equal to $0$.

\subsection{Dynamic risk measures}

We recall that static risk measures have been introduced to quantify \textit{now} the riskiness of financial instruments (or, more precisely, of their profits and losses or returns) ``expiring'' at a given time horizon $T$, while dynamic risk measures are developed to quantify such riskiness \textit{at any time} between now and the maturity $T$.
To be more precise, a \textit{static risk measure} on $L^{\infty}$ is a functional $\rho: L^{\infty}(\mathcal{F}_T) \to \mathbb{R}$ satisfying some further financially reasonable axioms, while a \textit{dynamic risk measure}
is a family $(\rho_t)_{t \in [0,T]}$, with $\rho_t:L^{\infty}(\mathcal{F}_T)\to L^{\infty}(\mathcal{F}_t)$, such that $\rho_0$ is a static risk measure and $\rho_T(X)=-X$ for any $X \in L^{\infty}(\mathcal{F}_T)$. Note that any $X \in L^{\infty}(\mathcal{F}_T)$ has to be understood as the profit and loss of a financial position (where positive values stand for profits, while negative for losses).

For a comprehensive literature on risk measures, we address the reader to \cite{ArDeEbHe99,delb,foellmer,puttingorder,bion1,BaEl,CerreiaVioglio,Frittelli2,det,ElRa,viag}, among many others.\medskip

Here below we collect a non exhaustive list of desirable axioms that may be assumed for (dynamic) risk measures:\medskip

\noindent - \textit{monotonicity}: If $X\geq Y$, with $X,Y\in L^{\infty}{(\mathcal{F}_T)}$, then $\rho_t(X)\leq \rho_t(Y)$ for any $t\in[0,T]$.\smallskip

\noindent - \textit{convexity}: $\rho_t(\alpha X + (1-\alpha)Y) \leq \alpha \rho_t(X) + (1-\alpha) \rho_t(Y)$ for any $\alpha \in [0,1], X,Y \in L^{\infty}(\mathcal{F}_T), t\in[0,T]$.\smallskip

\noindent - \textit{cash-additivity}: $\rho_t(X+m_t)=\rho_t(X)-m_t$ for any $m_t\in L^{\infty}{(\mathcal{F}_t)}, X\in L^{\infty}{(\mathcal{F}_T)}, t\in[0,T]$.\smallskip

\noindent - \textit{cash-subadditivity}: $\rho_t(X+m_t)\geq\rho_t(X)-m_t$ for any $m_t\in L^{\infty}_+{(\mathcal{F}_t)},X\in L^{\infty}{(\mathcal{F}_T)}, t\in[0,T]$.\smallskip

\noindent - \textit{time-consistency}: $\rho_s(X)=\rho_s(-\rho_t(X))$ for any $X\in L^{\infty}{(\mathcal{F}_T)}, 0\leq s\leq t \leq T$.\smallskip

\noindent - \textit{weak time-consistency}: $\rho_s(X)\leq \rho_s(-\rho_t(X))$ for any $X\in L^{\infty}(\mathcal{F}_T), 0\leq s\leq t\leq T$.\smallskip

\noindent - \textit{normalization}: $\rho_t(0)=0$ for any $t \in [0,T]$.\smallskip

\noindent - \textit{continuity from below (resp. above)}: for any sequence $(X_n)_{n\in\mathbb{N}}\subseteq L^{\infty}(\mathcal{F}_T)$ such that $X_n\uparrow X$ a.s. (resp. $X_n\downarrow X$ a.s.) with $X\in L^{\infty}(\mathcal{F}_T)$, it holds that $\rho_t(X_n)\downarrow\rho_t(X)$ a.s. (resp. $\rho_t(X_n)\uparrow\rho_t(X)$ a.s.).
\bigskip

\noindent In particular, a dynamic risk measure will be said to be convex if it satisfies monotonicity, convexity and normalization.\smallskip

(Decreasing) Monotonicity, convexity, normalization, and continuity from below/above are quite commonly assumed for risk measures. In particular, convexity is related to diversification of risk, while monotonicity to the interpretation of risk measures where the riskier is the financial position $X$, the greater is the value of $\rho(X)$. The riskiness of having zero profit and loss is set to be $0$ when normalization holds.
Cash-additivity, instead, guarantees that, if we add some cash (or some amount that is known at time $t$) to the financial position $X$, then the riskiness of the position is diminished exactly of the cash amount. Somehow, then, cash-additivity assumes zero interest rate and provides the financial interpretation of risk measures in terms of capital requirements (or margins). As argued in \cite{ElRa}, however, in presence of stochastic interest rate or ambiguity on interest rates it should be financially reasonable to replace cash-additivity with cash-subadditivity.

Finally, time-consistency and weak time-consistency provide the inter-temporal behavior of risk measures. While time-consistency means that at any time $s$ it is indifferent to evaluate the riskiness of a financial position $X$ in one step or in two (or more) steps backwards, weak time-consistency assumes that the riskiness measured in one step is lower than that evaluated in more steps.\smallskip

We recall from \cite{Bion} that monotonicity together with cash-additivity implies\medskip

\noindent - \textit{regularity}:
$\rho_t(X\mathbb{I}_A+Y\mathbb{I}_{A^C})=\rho_t(X)\mathbb{I}_A+\rho_t(Y)\mathbb{I}_{A^C}$ for any $t\in[0,T], A\in\mathcal{F}_t, X,Y\in L^{\infty}(\mathcal{F}_T)$.\medskip

\noindent It can be proved that also cash-subadditivity and monotonicity are enough to ensure regularity.\smallskip

Concerning dual representations of risk measures, it is well known that:\medskip

\noindent - any dynamic convex risk measure on  $L^{\infty}(\mathcal{F}_T)$ satisfying \textit{cash-additivity} and continuity from below has the following dual representation (see \cite{det}):
\begin{equation}
    \rho_t(X)=\essmax_{\mathbb{Q}_t\in\mathcal{Q}} \, \{\mathbb{E}_{\mathbb{Q}_t}[-X|\mathcal{F}_t]-c_t(\mathbb{Q}_t)\},
    \label{representationCARM}
\end{equation}
where, for any $t \in [0,T]$,
$$\mathcal{Q}\triangleq\{\mathbb{Q}_t \mbox{ probability measures on } (\Omega,\mathcal{F}): \mathbb{Q}_t\ll\mathbb{P} \mbox{ and } \mathbb{Q}_t|_{\mathcal{F}_t}=\mathbb{P}|_{\mathcal{F}_t}\}$$
and $c_t$ is a non-negative $\mathcal{F}_t$-measurable penalty function;\smallskip

\noindent- any dynamic convex risk measure  $L^{\infty}(\mathcal{F}_T)$ satisfying \textit{cash-subadditivity} and continuity from above has the following dual representation (see \cite{MastroRos}):
\begin{equation}
\rho_t(X)=\esssup_{(D_t,\mathbb{Q}_t)\in\mathcal{D}\times\mathcal{Q}}\{D_t\mathbb{E}_{\mathbb{Q}_t}[-X|\mathcal{F}_t]-c_t(D_t \mathbb{Q}_t)\},
    \label{representationSRM}
\end{equation}
where $$\mathcal{D}\triangleq\{(D_t)_{t\in[0,T]} \mbox{ adapted processes w.r.t. } ({\mathcal{F}_t})_{t\in[0,T]}: D_t \in [0,1] \, \mbox{ for every } t\in[0,T], \mbox{ a.s.}\}$$
can be interpreted as a set of discounting factors, while $c_t(D_t \mathbb{Q}_t)$ is a non-negative $\mathcal{F}_t$-measurable random variable which plays the role of a generalized penalty function.\medskip

Henceforth, we will always impose that the essential supremum in \eqref{representationSRM} is attained at some $(\bar{D}_t,\bar{\mathbb{Q}}_t)\in\mathcal{D}\times\mathcal{Q}$, i.e.
\begin{equation*}
(\bar{D}_t,\bar{\mathbb{Q}}_t)\in\argessmax_{(D_t,\mathbb{Q}_t)\in\mathcal{D}\times\mathcal{Q}} \, \{D_t\mathbb{E}_{\mathbb{Q}_t}[-X|\mathcal{F}_t]-c_t(D_t \mathbb{Q}_t)\}.
\end{equation*}
In such a case, $(\bar{D}_t,\bar{\mathbb{Q}}_t)$ is referred to as an optimal scenario and
\begin{equation}
    \rho_t(X)=\bar{D}_t \mathbb{E}_{\bar{\mathbb{Q}}_t}[-X|\mathcal{F}_t]-c_t(\bar{D}_t \bar{\mathbb{Q}}_t).
    \label{representation2}
\end{equation}

The following result guarantees that this holds true for cash-subadditive convex risk measures satisfying continuity from below, similarly as for the cash-additive case.

\begin{proposition}\label{essmaxCARM}
If $\rho_t:L^{\infty}(\mathcal{F}_T)\to L^{\infty}(\mathcal{F}_t)$ is a dynamic convex cash-subadditive risk measure satisfying continuity from below, then
\begin{equation}
\rho_t(X)=\essmax_{(D_t,\mathbb{Q}_t)\in \mathcal{D}\times\mathcal{Q}} \,\left\{\mathbb{E}_{\mathbb{Q}_t}\left[-D_tX\vert\mathcal{F}_t\right]-c_t(D_t\mathbb{Q}_t)\right\}
        =\mathbb{E}_{\bar{\mathbb{Q}}_t}\left[-\bar{D}_tX\vert\mathcal{F}_t\right]-c_t(\bar{D}_t\bar{\mathbb{Q}}_t)
        \label{CFBrepresentation}
\end{equation}
for some $(\bar{D}_t,\bar{\mathbb{Q}}_t)\in\mathcal{D}\times\mathcal{Q},$ where $c_t$ is the minimal penalty term, i.e.
\begin{equation}
    c_t(D_t\mathbb{Q}_t)\triangleq\esssup_{X\in L^{\infty}(\mathcal{F}_T)}\left\{D_t\mathbb{E}_{\mathbb{Q}_t}[-X\vert\mathcal{F}_t]-\rho_t(X)\right\}.
    \label{minpen}
\end{equation}
\end{proposition}

\begin{proof}
Let $\rho_t$ be a dynamic convex, cash-subadditive, and continuous from below risk measure. It is then easy to check that
$\rho_{0,t}=\mathbb{E}[\rho_t(X)]$ is a static risk measure with the same properties. According to the proof of Theorem 4.3 in \cite{ElRa} and to the representation of convex static risk measures that are continuous from below (see \cite{BaEl} or \cite{foellmer,Frittelli2} for a thorough discussion), it follows that
\begin{equation}
    \rho_{0,t}(X)=\max_{\mu\in\mathcal{S}}\left\{ \mathbb{E}_{\mu}(-X)-c_{0,t}(\mu) \right\}=\mathbb{E}_{\bar{\mu}}(-X)-c_{0,t}(\bar{\mu}),
\end{equation}
where $c_{0,t}$ is the minimal penalty function of $\rho_{0,t}$ and $\mathcal{S}$ is the set formed by all sub-probabilities $\mu$ on $(\Omega,\mathcal{F})$ that are absolutely continuous w.r.t. $\mathbb{P}$. We recall (see \cite{ElRa} for further details) that with sub-probability it is meant a measure $\mu: (\Omega,\mathcal{F}) \to [0,1]$ such that $0 \leq \mu(\Omega) \leq 1$.

From Proposition 6 in \cite{MastroRos1}, the sub-probability $\bar{\mu}$ can be decomposed as $\bar{\mu}=\bar{D}_t\bar{\mathbb{Q}}_t$ with $(\bar{D}_t,\bar{\mathbb{Q}}_t)\in\mathcal{D}\times\mathcal{Q}$ whenever $c_{0,t}(\bar{\mu})<+\infty$; moreover, $c_{0,t}(\bar{D}_t\bar{\mathbb{Q}}_t)=\mathbb{E}[c_t(\bar{D}_t\bar{\mathbb{Q}}_t)]$. Hence
\begin{equation}
\rho_{0,t}(X)=\mathbb{E}[\rho_t(X)]=\mathbb{E}[{\bar{D}_t\mathbb{E}_{\bar{\mathbb{Q}}_t}}[-X\vert\mathcal{F}_t]-c_t(\bar{D}_t\bar{\mathbb{Q}}_t)]. \label{eq: eqstar000}
\end{equation}
The arguments above imply that
$$\rho_t(X)=\esssup_{(D_t,\mathbb{Q}_t)\in\mathcal{D}\times\mathcal{Q}}\left\{D_t\mathbb{E}_{\mathbb{Q}_t}{[-X\vert\mathcal{F}_t]-c_t(D_t\mathbb{Q}_t)}\right\}\geq \bar{D}_t\mathbb{E}_{\bar{\mathbb{Q}}_t}[-X\vert\mathcal{F}_t]-c_t(\bar{D}_t\bar{\mathbb{Q}}_t),$$
hence
\begin{equation*}
\rho_t(X)-\Big(\bar{D}_t\mathbb{E}_{\bar{\mathbb{Q}}_t}[-X\vert\mathcal{F}_t]-c_t(\bar{D}_t\bar{\mathbb{Q}}_t)\Big)\geq 0
\end{equation*}
and, by \eqref{eq: eqstar000},
\begin{equation*}
\mathbb{E}\left[\rho_t(X)-\Big(\bar{D}_t\mathbb{E}_{\bar{\mathbb{Q}}_t}[-X\vert\mathcal{F}_t]-c_t(\bar{D}_t\bar{\mathbb{Q}}_t)\Big)\right]=0.
\end{equation*}
The thesis then follows.
\end{proof}

In full generality, we say that a cash-subadditive convex risk measure is representable if
\begin{equation}
\rho_t(X)= \esssup_{(D_t,\mathbb{Q}_t)\in\mathcal{D}'\times\mathcal{Q}} \left\{ \mathbb{E}_{\mathbb{Q}_t}[-D_tX|\mathcal{F}_t]-c_t(D_t \mathbb{Q}_t) \right\}
    \label{representation1}
\end{equation}
for some $D_t\in\mathcal{D}'$, where $\mathcal{D}'$ is formed by all (not necessarily adapted) stochastic processes $(D_t)_{t\in[0,T]}$ that are valued in $[0,1]$.

\subsection{Capital allocation rules}

In this section, we briefly recall the capital allocation problem, in a static and dynamic setting, and with an axiomatic approach (see, e.g., \cite{boonen,canna,lecturenotes,delnaut,Kalkbrener,KO2,Tasche,tsanakas1}). The first axiomatic approach is given by Kalkbrener \cite{Kalkbrener} is a static setting, while the dynamic setting has been firstly studied in \cite{KO} for risk measures induced via BSDEs and in the case of gradient allocation. Nevertheless, the first organic work about (general) dynamic capital allocations is provided in Mastrogiacomo and Rosazza Gianin \cite{MastroRos} where the approach of \cite{Kalkbrener} is generalized to a dynamic setting.\smallskip

According to the classical literature, a static capital allocation rule consists of deciding how to divide the capital requirement (or margin) of an aggregate risky position $X$ into the sub-portfolios (also called business lines) which $X$ is formed of, according to some financially sound criteria. To be more concrete, given a static risk measure $\rho$, an aggregate position $X$ and its business lines $X_1, ...,X_n$ (hence such that $X= X_1+...+X_n$), a capital allocation rule prescribes how to share $\rho(X)$ into $X_1,...,X_n$ by allocating the capital $k_i$ to $X_i$ such that $\rho(X)=\sum_{i=1}^{n}k_i$ (when the so called full allocation of the capital allocation principle is fulfilled).

Starting from those considerations, Kalkbrener \cite{Kalkbrener} introduced the notion of capital allocation rule (CAR, for short) associated to a risk measure $\rho$ as a functional $\Lambda:L^{\infty}(\mathcal{F}_T)\times L^{\infty}(\mathcal{F}_T) \to \mathbb{R}$ such that $\Lambda(X,X)=\rho(X)$ for any $X \in L^{\infty}$.  In other words, the capital to be allocated to the whole aggregate portfolio $X$ (as a stand-alone portfolio) coincides with the riskiness of $X$ evaluated via the static risk measure $\rho$. Furthermore,  $\Lambda(Y,X)$ should be understood as the capital to be allocated to $Y$ as a sub-portfolio of $X$.
\smallskip

We now recall the definition of a dynamic CAR. 

\begin{definition}[see \cite{MastroRos}]
Let $(\rho_t)_{t\in[0,T]}$ be a dynamic risk measure.

A family of functionals $\Lambda_t:L^{\infty}(\mathcal{F}_T)\times L^{\infty}(\mathcal{F}_T) \to  L^{\infty}(\mathcal{F}_t)$, with $t\in[0,T]$, is said to be a dynamic CAR associated to the risk measure ${(\rho_t)}_{t\in[0,T]}$ if $$\Lambda_t(X,X)=\rho_t(X), \mbox{ for any } t \in [0,T] \mbox{ and }  X\in L^{\infty}(\mathcal{F}_T).$$
Instead, the family ${(\Lambda_t)}_{t\in[0,T]}$ is called \textit{audacious} CAR if $\Lambda_t(X,X)\leq\rho_t(X)$, for any $t\in[0,T]$ and $X\in L^{\infty}(\mathcal{F}_T)$.
\label{DEFCAR}
\end{definition}

In other words, while a dynamic CAR requires that the capital to be allocated to $X$ as a stand-alone portfolio should be equal to the margin $\rho_t(X)$, an audacious one only asks to allocate not more than the margin (see \cite{centrone,MastroRos} and the references therein for a discussion):

In the sequel, when no confusion may occur, we identify the family of functionals ${(\Lambda_t)}_{t\in[0,T]}$ with an element $\Lambda_t$.
Here below, we provide a (non exhaustive) list of axioms that could be sometimes imposed to a dynamic CAR. See \cite{MastroRos} and the references therein for more details.\medskip

\noindent - \textit{monotonicity:} if $X\geq Y$ (with $X,Y\in L^{\infty}(\mathcal{F}_T)$), then $\Lambda_t(X,Z)\leq \Lambda_t(Y,Z)$ for any $t\in [0,T]$, $Z \in L^{\infty}(\mathcal{F}_T)$.

\noindent - \textit{no-undercut}: $\Lambda_t(X,Y)\leq \Lambda_t(X,X)$ for any $t\in[0,T]$, $X,Y\in L^{\infty}(\mathcal{F}_T)$.

\noindent - \textit{cash-additivity}: $\Lambda_t(X+m_t,Y+m_t)=\Lambda_t(X,Y)-m_t$, for any $t\in[0,T]$, $m_t\in L^{\infty}(\mathcal{F}_t),  X,Y \in L^{\infty}(\mathcal{F}_T)$.

\noindent - \textit{cash-subadditivity}: $\Lambda_t(X+m_t,Y+m_t)\geq\Lambda_t(X,Y)-m_t$ for any $t\in[0,T]$, $m_t\in L_+^{\infty}(\mathcal{F}_t),  X,Y \in L^{\infty}(\mathcal{F}_T)$.

\noindent - \textit{1-cash-additivity}: $\Lambda_t(X+m_t,Y)=\Lambda_t(X,Y)-m_t$ for any $t\in[0,T]$, $m_t\in L^{\infty}(\mathcal{F}_t),  X,Y \in L^{\infty}(\mathcal{F}_T)$.

\noindent - \textit{1-cash-subadditivity}: $\Lambda_t(X+m_t,Y)\geq\Lambda_t(X,Y)-m_t$ for any $t\in[0,T]$, $m_t\in L^{\infty}_+(\mathcal{F}_t), X,Y \in L^{\infty}(\mathcal{F}_T)$.

\noindent - \textit{normalization}: $\Lambda_t(0,X)=0$ for any $t\in[0,T], X\in L^{\infty}(\mathcal{F}_T)$.

\noindent - \textit{full allocation}: $\sum_{i=1}^{n} \Lambda_t \left(X_i,X\right)=\Lambda_t(X,X)$ for any $t \in [0,T]$ and $(X_i)_{i=1,\dots, n},X\in L^{\infty}(\mathcal{F}_T)$ with $\sum_{i=1}^n X_i=X$.\medskip

We shortly recall the financial interpretation and motivation of the previous axioms (see, for instance, \cite{delnaut,centrone,Kalkbrener,MastroRos,Tasche,tsanakas1} for a deeper discussion). Firstly, monotonicity means that the riskier a position is, the greater is the capital allocated needed to cover its risk. No-undercut property guarantees that the risk of any sub-portfolio $X$ of $Y$ cannot exceed the total risk of portfolio $X$ seen as a stand-alone portfolio, implying therefore that there is no incentive to split $X$ from $Y$ (in the terminology of Tsanakas \cite{tsanakas}). Cash-additivity conveys the principle according to which any addition of 'known' (or, better, $\mathcal{F}_t$-measurable) risk $m_t$ to both the sub-unit $X$ and the position $Y$ yields a decrement of the risk exactly equal to $m_t$. Cash-subadditivity substitutes cash-additivity in the case of the presence of ambiguous rate interest in the market (in the same spirit of El Karoui and Ravanelli \cite{ElRa} for risk measures). Similar interpretations hold for 1-cash-additivity and 1-cash-subadditivity. Normalization requires that the risk associated to an initial null sub-portfolio remains equal to zero at every time $t\in[0,T]$. The full allocation requirement seems to be quite natural in the context of capital allocation rules because it implies that the total contribution to the risk of the sum of every sub-portfolio is exactly equal to the risk of the overall position. However, as from \cite{Kalkbrener}, full allocation and no-undercut properties could be incompatible for capital allocations rules when the underlying risk measure is strictly convex and not coherent. Nevertheless, when a firm is only interested in monitoring its financial activities, full allocation can be replaced by one of the following axioms:\medskip

\noindent - \textit{sub-allocation}: for any $t \in [0,T]$,
$$\Lambda_t \left(X,X \right) \geq \sum_{i=1}^{n} \Lambda_t \left(X_i,X \right)$$
holds for any $(X_i)_{i=1,\dots, n},X\in L^{\infty}(\mathcal{F}_T)$ with $\sum_{i=1}^n X_i=X$.

\noindent - \textit{weak-convexity}: for any $t \in [0,T]$,
$$\Lambda_t \left(\ds\sum_{i=1}^n a_iX_i,X \right)\leq \ds\sum_{i=1}^n a_i\Lambda_t(X_i,X)$$
for any $X, X_i\in L^{\infty}(\mathcal{F}_T),a_i\in[0,1]$, for $i=1,\dots,n$, with $\sum_{i=1}^na_i=1$ and $\sum_{i=1}^na_iX_i=X$.\medskip

In particular, for sub-allocation, the difference between the capital allocated to cover the risk of all sub-portfolios and the capital stored to insure the risk of the full portfolio can be understood as an extra-security margin justified, e.g., by some extra-costs shared among all the sub-portfolios (see also the discussion in \cite{brun,centrone,canna}).
Weak-convexity, instead, reflects the diversification of risk that encourages the sharing of risk between different assets.

Finally, as for risk measures, also for dynamic capital allocations rules it is relevant to investigate time-consistency. The following two formulations of time-consistency have been introduced in \cite{MastroRos}:\medskip

\noindent - \textit{1-time-consistency}: $\Lambda_s(-\Lambda_t(X,Y),Y)=\Lambda_s(X,Y)$ for any $0\leq s \leq t \leq T,  X,Y\in L^{\infty}(\mathcal{F}_T)$.\smallskip

\noindent - \textit{time-consistency}:
   $\Lambda_s(-\Lambda_t(X,Y),-\Lambda_t(Y,Y))=\Lambda_s(X,Y)$ for any $0\leq s \leq t \leq T, X,Y\in L^{\infty}(\mathcal{F}_T)$.\medskip

Such formulations guarantee that, roughly speaking, it is indifferent to compute the capital allocation with a one-step procedure or proceeding backwards in time with multiple steps (see \cite{MastroRos} for a detailed discussion).

\subsection{BSDEs and BSVIEs}

We shortly review now the main definitions and results on Backward Stochastic Differential Equations (BSDEs) and Backward Stochastic Volterra Integral Equations (BSVIEs) used in the sequel. On the probability space $(\Omega,\mathcal{F},\mathbb{P})$, consider a $n$-dimensional Brownian motion $(B_t)_{t\in[0,T]}$ with its augmented natural filtration $(\mathcal{F}_t)_{t\in[0,T]}$.

It is well known (see \cite{revuz-yor}) that, in a Brownian setting, for any probability measure $\mathbb{Q} \ll\mathbb{P}$ there exists a predictable process $(q_t)_{t\in[0,T]}$ such that:
$$\mathbb{E}\left[\frac{d\mathbb{Q}}{d\mathbb{P}}\Bigg| \mathcal{F}_t\right]=\mathcal{E}(q\cdot W)_t \triangleq \exp\left(-\frac{1}{2}\int_0^t|q_s|^2ds-\int_0^tq_sdB_s\right),$$
where $|\cdot|$ stands for the Euclidean norm both in $\mathbb{R}$ and in $\mathbb{R}^n$. Henceforth, any probability measure $\mathbb{Q}\ll\mathbb{P}$ will be identified with its Radon-Nikodym density $\frac{d\mathbb{Q}}{d\mathbb{P}}$.

We fix now some notations we will use in the following. Let $p\in[1,+\infty)$, $0\leq\tau<T$ and $\Delta_{\tau}\triangleq\{(t,s)\in[\tau,T]^2: t\leq s\}$, with the convention  $\Delta_0\triangleq\Delta$. Let us denote
\begin{eqnarray*}
	\hspace{-0.7cm}&&L^p(\Omega,\mathcal{F}_t;L^q(\tau,T))\triangleq\left\{\phi:\Omega\times[\tau,T]\to \mathbb{R}: \phi \mbox{ is } \mathcal{F}_t\otimes\mathcal{B}([\tau,T])\mbox{-meas. and }  \mathbb{E}\left(\int_{\tau}^T|\phi_s|^qds\right)^{\frac{p}{q}}\hspace{-0.1cm}<+\infty\right\}, \\
	\hspace{-0.7cm}&&L^p(\Omega,\mathbb{F};L^q(\tau,T))\triangleq\left\{\phi:\Omega\times[\tau,T]\to \mathbb{R}: \phi \mbox{ is }(\mathcal{F}_t)_{t\in[\tau,T]}\mbox{-adapted and } \mathbb{E}\left(\int_\tau^T|\phi_s|^qds\right)^{\frac{p}{q}}\hspace{-0.1cm}<+\infty \right\},
\end{eqnarray*}
while let $L^q((\tau,T);L^p(\Omega,\mathbb{F};L^m(\tau,T)))$ denote the space of all stochastic processes $\phi:\Omega\times\Delta_{\tau}\to\mathbb{R}$ such that, for any $t\in[\tau,T]$,
$$
\phi(t,\cdot)\mbox{ is } (\mathcal{F}_m)_{m\in[t,T]}-\mbox{adapted and } \int_{\tau}^T\Big\{\mathbb{E}\Big(\int_{t}^T|\phi(t,s)|^mds\Big)^{p/m}\Big\}^{q/p}dt <+\infty.
$$
Furthermore, let us denote
\begin{eqnarray*}
	\hspace{-0.6cm}&&\mathcal{H}^{p,q}(\tau,T)\triangleq L^p(\Omega,\mathbb{F};L^q(\tau,T))\times L^q((\tau,T),\mathcal{F}_t;L^p(\Omega,\mathbb{F};L^2(\tau,T))),
	\\
	\hspace{-0.6cm}&&\mbox{BMO}({\mathbb{P}})\triangleq\left\{\phi\in L^2(\Omega,\mathbb{F};L^2(\tau,T)): \ \exists C>0 \mbox{ s.t., } \forall t\in[\tau,T], \ \mathbb{E}\left[\left.\int_t^T|\phi_s|^2ds \right|\mathcal{F}_t\right]\leq C \mbox{ a.s.}\right\} \\
		\hspace{-0.6cm}&&L^{p}((0,T);\mbox{BMO}({\mathbb{P}}))\triangleq \left\{\phi:\Omega\times\Delta_{\tau}\to\mathbb{R}: \int_{\tau}^T\left\{\left\|\sup_{r\in[t,T]}\mathbb{E}\left[ \left. \int_r^T|\phi(t,s)|^2ds \right| \mathcal{F}_r\right]\right\|_{\infty}\right\}^pdt <+\infty \right\}.	
\end{eqnarray*}

We can also include the case $p=+\infty$, adapting our definitions in the usual way.
To simplify the notation, for $p=q$ we shorten $L^p(\Omega,\mathcal{F}_t;L^p(\tau,T))\triangleq L^p((\tau,T);\mathcal{F}_t)$ and analogously for the other cases.
We observe that if $\phi\in\mbox{BMO}(\mathbb{P})$ then $$\left\|\sup_{t\in[\tau,T]}\mathbb{E}\left[\left. \int_t^T|\phi_s|^2ds \right|\mathcal{F}_t\right]\right\|_{\infty}<+\infty.$$

\subsubsection{BSDEs and related cash-subaddditive risk measures}

Following \cite{EPQ,kob,pardoux,peng-97} and the references therein, a Backward Stochastic Differential Equation (BSDE) is an equation of the following form
\begin{equation}
    Y_t=X+\int_t^Tg(s,Y_s,Z_s)ds-\int_t^TZ_sdB_s,
    \label{BSDE}
\end{equation}
where $X \in L^2(\mathcal{F}_T)$ is the final condition, $g:\Omega\times[0,T]\times\mathbb{R}\times\mathbb{R}^n \to \mathbb{R}$ is called driver of the BSDE, while $(Y_t,Z_t)_{t\in[0,T]}$ are the unknown processes where $Y$ is $\mathbb{R}$-valued and $Z$ is $\mathbb{R}^n$-valued. For simplicity, the dependence on $\omega\in\Omega$ is conventionally often omitted.

We recall the following existence and uniqueness result.
\begin{theorem}[see \cite{pardoux,EPQ}]\label{thm: exist uniq}
Let $X\in L^{2}(\mathcal{F}_T)$ and suppose that $g:\Omega\times[0,T]\times\mathbb{R}\times\mathbb{R}^n\to\mathbb{R}$ satisfies the following assumptions:
\begin{itemize}
    \item[a)] Uniformly Lipschitz: there exists $C>0$ such that $d\mathbb{P}\otimes dt$-a.s.
    $$|g(\omega,t,y_1,z_1)-g(\omega,t,y_2,z_2)|\leq C(|y_1-y_2|+|z_1-z_2|) \quad \mbox{ for any } (y_1,z_1), (y_2,z_2)\in\mathbb{R}\times\mathbb{R}^n;$$
    \item[b)] $g$ is $\mathcal{F}_T\times\mathcal{B}([0,T]\times\mathbb{R}\times\mathbb{R}^n)$-measurable and $g(t,y,z)$ is an $(\mathcal{F}_t)_{t\in[0,T]}$-progressively measurable stochastic process for any $(y,z)\in\mathbb{R}\times\mathbb{R}^n$;
    \item[c)] $g(\cdot,0,0)\in L^{2}((0,T),\mathbb{F})$.
\end{itemize}
Then there exists a unique pair of stochastic processes $(Y_t,Z_t)_{t\in[0,T]}\in L^2((0,T),\mathbb{F})\times L^2((0,T),\mathbb{F})$ satisfying \eqref{BSDE} in It\^o sense. Furthermore, $Y_t$ has continuous paths.
\label{EUBSDE}
\end{theorem}

As argued in \cite{Bion}, if the final condition $X\in L^{\infty}(\mathcal{F}_T)$ and $g(\cdot,0,0)\equiv 0$ then we have the further regularity $(Y_t,Z_t)\in{L^{\infty}((0,T),\mathbb{F})}\times\mbox{BMO}(\mathbb{P})$. We need some results about further regularity of the solution when $g(\cdot,0,0)\not\equiv0$. Furthermore, the quadratic case has been studied in \cite{kob} and, later, in \cite{zheng}. In particular, \cite{zheng} has shown that if $g(\cdot,0,0)$ is deterministic and $\int_t^T|g(s,0,0)|ds<+\infty$ then the solution $(Y_t,Z_t)\in L^{\infty}((0,T);\mathbb{F})\times\mbox{BMO}(\mathbb{P})$. Under Lipschitz assumptions, we can prove the same thesis weakening the hypotheses on $g(\cdot,0,0)$. This result will be useful in the following, for what concerns the theory of BSVIEs.

\begin{proposition}
	\label{regularityBSDE}
	Let $X\in L^{\infty}(\mathcal{F}_T)$ and suppose that $g:\Omega\times[0,T]\times\mathbb{R}\times\mathbb{R}^n\to\mathbb{R}$ satisfies assumptions a) and b) of Theorem \ref{thm: exist uniq} plus c') $g(\cdot,0,0)\in\mbox{BMO}(\mathbb{P})$ (called 'standard assumptions' henceforth).

	Then there exists a unique pair of stochastic processes $(Y_t,Z_t)_{t\in[0,T]}\in L^{\infty}((0,T),\mathbb{F})\times \mbox{BMO}(\mathbb{P})$ satisfying \eqref{BSDE} in It\^o sense. Furthermore, $Y_t$ has continuous paths.
\end{proposition}

\begin{proof}
	By Theorem \ref{EUBSDE} there exists a unique solution $(Y_t,Z_t)_{t\in[0,T]}\in L^2((0,T),\mathbb{F})\times L^2((0,T),\mathbb{F})$, and $Y_t$ has continuous paths. We only need to prove the further regularity. By Proposition 3 of \cite{fan} we obtain:
	\begin{equation}
		\mathbb{E}\left[\sup_{m\in[t,T]}|y_m|^2\Big|\mathcal{F}_t\right]+\mathbb{E}\left[\int_t^T|Z_s|^2ds\Big|\mathcal{F}_t\right]\leq C\left(\mathbb{E}\left[|X|^2 | \mathcal{F}_t\right]+\mathbb{E}\left[\left. \int_t^T|g(s,0,0)|^2ds\right|\mathcal{F}_t\right]\right).
		\label{fan3}
	\end{equation}
As far as $Z_t$ regularity is concerned, by equation \eqref{fan3} and by taking the sup over $t\in[0,T]$ and $\omega\in\Omega$ we get:
\begin{equation}
\hspace{-0.4cm}\left\|\sup_{t\in[0,T]}\mathbb{E}\left[\left. \int_t^T|Z_s|^2ds\right|\mathcal{F}_t\right]\right\|_{\infty}\leq C\left(\|X\|_{\infty}^2+\left\|\sup_{t\in[0,T]}\mathbb{E}\left[\left. \int_t^T|g(s,0,0)|^2ds\right|\mathcal{F}_t\right]\right\|_{\infty}\right)\leq C<+\infty,
\label{estimateZ}
\end{equation}
where here and for the rest of the proof $C>0$ is a constant that varies from line to line and where the last inequality follows from $g(\cdot,0,0)\in\mbox{BMO}(\mathbb{P})$. Hence, $Z\in\mbox{BMO}(\mathbb{P})$. For what concerns $Y_t$ we have:
\begin{equation*}
\mathbb{E}\left[Y_t\big|\mathcal{F}_t\right]=Y_t=\mathbb{E}\left[X\big|\mathcal{F}_t\right]+\mathbb{E}\left[\left. \int_t^Tg(s,Y_s,Z_s)ds\right|\mathcal{F}_t\right].
\end{equation*}
By Jensen's inequality, 
Lipschitz assumption and inequalities above, it then follows that
\begin{align*}
	|Y_t|^2&\leq C\left(\|X\|_{\infty}^2+\mathbb{E}\left[\int_t^T\left(|g(s,0,0)|^2+|Z_s|^2\right)ds+\sup_{m\in[t,T]}|y_m|^2\Big|\mathcal{F}_t\right]\right) \\ &\leq C\left(\|X\|_{\infty}^2+\mathbb{E}\left[\int_t^T|g(s,0,0)|^2\Big|\mathcal{F}_t\right]\right).
\end{align*}
Taking the supremum, we then obtain:
\begin{equation}
	\Big\|\sup_{t\in[0,T]}|Y_t|\Big\|_{\infty}^2\leq C\left(\|X\|_{\infty}+\left\|\sup_{t\in[0,T]}\mathbb{E}\left[\int_t^T|g(s,0,0)|^2ds\Big|\mathcal{F}_t\right]\right\|_{\infty}\right)\leq C<+\infty.
\label{estimateY}
\end{equation}
Thus, $Y_t\in L^{\infty}((0,T),\mathbb{F})$ and the proof is concluded.
\end{proof}

The relation between BSDEs and dynamic risk measures has been deeply studied in the literature. We address an interested reader to \cite{viag,peng,BaEl,ElRa,penalty}, among others. In particular, concerning cash-subadditive risk measures, El Karoui and Ravanelli \cite{ElRa} proved that if the driver $g$ satisfies the standard assumptions (or a quadratic growth assumption), is convex in $(y,z)$ and decreasing in $y$, then $\rho_t(X) \triangleq Y_t^{-X}$, with $Y_t^{-X}$ being the first component of the solution of the BSDE
\begin{equation}
    Y_t^{-X}=-X+\int_t^Tg(s,Y_s,Z_s)ds-\int_t^T Z_sdB_s, \quad X\in L^{\infty}(\mathcal{F}_T),
\label{BSDEcsam}
\end{equation}
is a dynamic convex, cash-subadditive, and time-consistent risk measure. Furthermore, $\rho_t$ admits the following dual representation:
\begin{equation}
\rho_t(X)=\esssup_{(\beta_t,\mu_t)\in \mathcal{A}}\mathbb{E}_{\mathbb{Q}_t^{\mu}}\left[\left.e^{-\int_t^T\beta_sds}(-X)-\int_t^Te^{-\int_t^s\beta_udu} \, G(s,\beta_s,\mu_s)ds \right|\mathcal{F}_t\right],
\label{representation}
\end{equation}
where
\begin{equation*}
\mathbb{E}\left[\frac{d\mathbb{Q}^{\mu}_t}{d\mathbb{P}}\Bigg| \mathcal{F}_t\right]=\mathcal{E}(\mu\cdot B)_t,
\end{equation*}
$G$ is the Fenchel transformation (or convex conjugate) of $g$, i.e.
\begin{equation*}
G(t,\beta,\mu)\triangleq\ds\sup_{(y,z)\in\mathbb{R}\times\mathbb{R}^n}\left\{-\beta y-\langle \mu, z\rangle - g(t,y,z)\right\},
\end{equation*}
and
\begin{equation*}
\mathcal{A}\triangleq\left\{(\beta_t,\mu_t)_{t\in[0,T]} \  \mathcal{F}_t\mbox{-adapted}: G(t,\beta_t,\mu_t)<+\infty,  \beta_t\geq 0, 0\leq\beta_t+|\mu_t|\leq C  \mbox{ for any } t\in[0,T] \right\}.
\end{equation*}

By combining Theorems 7.5 in \cite{ElRa} and Proposition 3.6 in \cite{penalty}, it can be easily checked that if the driver $g$ satisfies the standard assumptions and convexity, then $G(t,\beta_t,\mu_t)=+\infty$ if $\beta_t<0$ or $\beta_t+|\mu_t|\not\in [0,C]$. Moreover, the $\esssup$ in \eqref{representation} is attained for some $(\Bar{\beta}_t,\Bar{\mu}_t)_{t\in[0,T]}\in\mathcal{A}$. Hence
\begin{equation}
    \rho_t(X)=\mathbb{E}_{\mathbb{Q}_t^{\bar{\mu}}}\left[ \left. e^{-\int_t^T\bar{\beta}_sds}(-X)-\int_t^Te^{-\int_t^s\bar{\beta}_udu} \, G(s,\bar{\beta}_s,\bar{\mu}_s)ds \right|\mathcal{F}_t\right].
    \label{BSDEmax}
\end{equation}

In particular, the following result deals with the sublinear and cash-subadditive case. We recall that with sublinear it is meant the case where both subadditivity and positive homogeneity are fulfilled. In terms of the driver $g$ this corresponds to\medskip

\noindent - positive homogeneity in $(y,z)$: $$ d\mathbb{P}\times dt\mbox{-a.s., } \ g(t,\alpha y,\alpha z)=\alpha g(t,y,z) \mbox{ for any } \alpha\geq 0,  (y,z)\in\mathbb{R}\times\mathbb{R}^n;$$

\noindent - subadditivity in $(y,z)$:
    $$ d\mathbb{P}\times dt\mbox{-a.s., } \  g(t,y_1+y_2,z_1+z_2)\leq g(t,y_1,z_1)+g(t,y_2,z_2) \mbox{ for any } (y_1,z_1),(y_2,z_2)\in\mathbb{R}\times\mathbb{R}^n.$$
\smallskip

\begin{lemma}\label{coherentRM}
Suppose that the driver $g$ is decreasing in $y$ and satisfies the standard assumptions. If, in addition, $g$ is sublinear,
then the dynamic risk measure $\rho_t(X)=Y_t^{-X}$ induced by \eqref{BSDEcsam} is sublinear and admits the dual representation:
\begin{equation}
    \rho_t(X)=\esssup_{(\beta_t,\mu_t)\in\mathcal{A}}\mathbb{E}_{\mathbb{Q}_t^{\mu}}\left[e^{-\int_t^T\beta_sds}(-X)\Big|\mathcal{F}_t\right]=\mathbb{E}_{\mathbb{Q}_t^{\bar{\mu}}}\left[e^{-\int_t^T\bar{\beta}_sds}(-X)\Big|\mathcal{F}_t\right], \mbox{ \  } X\in L^{\infty}(\mathcal{F}_T),
    \label{BSDEsublinear}
\end{equation}
where $(\bar{\beta},\bar{\mu}) \in \mathcal{A}$ is an optimal scenario.
\end{lemma}

\begin{proof}
If $g$ is sublinear then convexity is automatically satisfied, so the dual representation of $\rho_t$ provided in equations \eqref{representation} and \eqref{BSDEmax} holds for any $X\in L^{\infty}(\mathcal{F}_T)$. Moreover, from Theorem 2.3.1(V) of \cite{zali} it follows that the convex conjugate of a positively homogeneous function is $0$ in its domain (i.e. $G(t,\beta,\mu)<+\infty$ implies $G(t,\beta,\mu)=0$).
Finally, sublinearity of $\rho_t$ can be easily checked.
\end{proof}

\subsubsection{BSVIEs}

somehow related to BSDEs, the family of Backward Stochastic Volterra Integral Equations (BSVIEs) was studied for the first time in Yong \cite{Yong2006}. See also \cite{hu-oksendal,CTBSVIE} and, for the applications to risk measures, \cite{Nacira,DN-RG,Yong2007}. In the sequel, we focus on the following (special form of a) BSVIE:
\begin{equation}
    Y(t)=\varphi(t)+\int_t^T g(t,s,Y(s),Z(t,s))ds-\int_t^T Z(t,s)dB_s,
    \label{BSIVE}
\end{equation}
where $\varphi(t)$ and $g$ are given, while the pair $(Y(\cdot),Z(\cdot,\cdot))$ is unknown. Note that $\varphi (t) \in L^2(\mathcal{F}_T)$ is not necessarily $\mathcal{F}_t$-measurable but varies with $t$.

As in the case of standard assumptions for BSDEs, we introduce an analogous for BSVIEs (see, for instance, \cite{Yong2007}). We remind that $\Delta\triangleq\{(t,s)\in[0,T]^2: t\leq s\}$. In the sequel, $g:\Omega\times\Delta\times\mathbb{R}\times\mathbb{R}^n\to\mathbb{R}$ will be required to satisfy the following standard assumptions:\medskip

\noindent $i)$ $g$ is $\mathcal{F}_T\otimes\mathcal{B}\left(\Delta\times\mathbb{R}\times\mathbb{R}^n\right)$-measurable and $g(t,s,y,z)$ is $(\mathcal{F}_t)_{t\in[0,T]}$-progressively measurable on $[t,T]$ for any $(t,y,z)\in [0,T)\times\mathbb{R}\times\mathbb{R}^n$; \smallskip

\noindent $ii)$ Uniformly Lipschitz: there exists $L>0$ such that
$$|g(t,s,y_1,z_1)-g(t,s,y_2,z_2)|\leq L(|y_1-y_2|+|z_1-z_2|)
$$
for any $(t,s)\in\Delta,(y_1,z_1),(y_2,z_2)\in\mathbb{R}\times\mathbb{R}^n$; \smallskip

\noindent $iii)$ $\mathbb{E} \left[ \int_0^T\left(\int_t^T|g(t,s,0,0)|ds\right)^2dt \right]<+\infty.$ \medskip

We recall from \cite{wellpos} that a pair $(Y(\cdot),Z(\cdot,\cdot))\in\mathcal{H}^1(0,T)$ is a solution of the BSVIE \eqref{BSIVE} if it solves equation \eqref{BSIVE} for almost all $t\in[0,T]$ in the usual It\^o sense.

\begin{theorem}[Existence and Uniqueness - see Yong \cite{Yong2006}]
	Let the driver $g:\Omega\times\Delta\times\mathbb{R}\times\mathbb{R}^n\to\mathbb{R}$ satisfy the standard assumptions. Then, for any $\varphi\in L^2((0,T),{\mathcal{F}_T};\mathbb{R})$, the BSVIE \eqref{BSIVE} admits a unique solution in $\mathcal{H}^2(0,T)$.
\label{EUBSVIE}
\end{theorem}
\noindent In the following we are mostly interested in BSVIEs of the form:
\begin{equation}
    Y(t)=\varphi(t)+\int_t^T g(t,s,Z(t,s))ds-\int_t^T Z(t,s)dB_s, \ t\in[0,T],
    \label{rBSVIE}
\end{equation}
when $g$ does not depend on $y$ and $\varphi\in L^{\infty}((0,T),\mathcal{F}_T)$. We show in the next theorem that the solution $(Y(\cdot),Z(\cdot,\cdot))$ of the BSVIE above has some further regularity.

\begin{theorem}\label{LinfBSVIE}
Let $g:\Omega\times\Delta\times\mathbb{R}^n\to\mathbb{R}$ be independent of $y$ and satisfy assumptions i) and ii) above plus the further constraint $g(\cdot,\cdot,0,0)\in L^{\infty}((0,T);\mbox{BMO}(\mathbb{P}))$, i.e.
$$\sup_{t\in[0,T]}\left\|\sup_{r\in[t,T]}\mathbb{E}\left[\int_r^T|g(t,s,0)|^2ds\Big|\mathcal{F}_r\right]\right\|_{\infty}<+\infty.$$
Let $\varphi\in L^{\infty}((0,T),\mathcal{F}_T)$.
Then there exists a unique solution to BSVIE \eqref{rBSVIE} such that
$$(Y(\cdot),Z(\cdot,\cdot))\in L^{\infty}((0,T),\mathbb{F})\times L^{\infty}((0,T);\mbox{BMO}(\mathbb{P})).$$
\end{theorem}

\begin{proof}
Existence and uniqueness of the solution in $\mathcal{H}^2(0,T)$ are due to Theorem \ref{EUBSVIE}. Thus, we only need to prove the further regularity of this solution. To this aim, we consider the following family of BSDEs parameterized by $t\in[0,T]$:
\begin{equation}
	    \eta(r;t)=\varphi(t)+\int_r^Tg(t,s,\zeta(s;t))ds-\int_r^T\zeta(s;t)dB_s, \ \ r\in[t,T].
\end{equation}
By Proposition \ref{regularityBSDE}, for all fixed $t\in[0,T]$ the previous BSDE admits a unique adapted solution $(\eta(t,\cdot),\zeta(t,\cdot))\in L^{\infty}((t,T),{\mathbb{F}})\times \mbox{BMO}(\mathbb{P})$. Furthermore, by \eqref{estimateZ} we have the estimate:
\begin{equation*}
\left\|\sup_{r\in[t,T]}\mathbb{E}\left[\int_r^T|\zeta(s,t)|^2ds\Big|\mathcal{F}_r\right]\right\|_{\infty}\leq C_1 \hspace{-0.1cm} \left(\|\varphi(t)\|_{\infty}^2+\left\|\sup_{r\in[t,T]}\mathbb{E}\left[\int_r^T|g(t,s,0)|^2ds\Big|\mathcal{F}_r\right]\right\|_{\infty}\right)\hspace{-0.1cm} \leq C_2 \hspace{-0.1cm}<+\infty
\end{equation*}
for some constants $C_1,C_2>0$ that do not depend on $(t,s)\in\Delta$. It then follows that
\begin{eqnarray*}
	&&\sup_{t\in[0,T]}\left\|\sup_{r\in[t,T]}\mathbb{E}\left[\left.\int_r^T|\zeta(s,t)|^2ds\right|\mathcal{F}_r\right]\right\|_{\infty} \notag \\
	&&\leq C_3\left(\sup_{t\in[0,T]}\|\varphi(t)\|_{\infty}^2+\sup_{t\in[0,T]}\left\|\sup_{r\in[t,T]}\mathbb{E}\left[\left.\int_r^T|g(t,s,0)|^2ds\right|\mathcal{F}_r\right]\right\|_{\infty}\right)\leq C_4<+\infty,
	\label{estimates1}
\end{eqnarray*}
for some constants $C_3,C_4>0$, where the last inequality follows from $\varphi\in L^{\infty}((0,T),\mathcal{F}_T)$ and from $g(\cdot,\cdot,0,0)\in L^{\infty}((0,T);\mbox{BMO}(\mathbb{P}))$. 
 For any $t\in[0,T]$, \eqref{estimateY} also gives the estimate:
\begin{equation*}
	\Big\|\sup_{r\in[t,T]}|\eta(r;t)|\Big\|_{\infty}^2\leq C_5\left(\|\varphi(t)\|_{\infty}+\left\|\sup_{r\in[t,T]}\mathbb{E}\left[\int_r^T|g(t,s,0)|^2ds\Big|\mathcal{F}_t\right]\right\|_{\infty}\right)\leq C_6<+\infty
\end{equation*}
for some $C_5,C_6 >0$. Taking the supremum over $t\in[0,T]$, we can conclude $\eta\in L^{\infty}((0,T);L^{\infty}((t,T),\mathbb{F}))$.
The thesis follows by considering
$$Y(t)\triangleq\eta(t;t), \ Z(t,s)\triangleq\zeta(s;t), \quad \mbox{ for } (t,s)\in\Delta.$$
\end{proof}

We are now ready to present a version of the Comparison Theorem for BSVIEs. This result will be used in the ensuing parts of the paper. More details about Comparison Theorems for BSVIEs can be found in \cite{CTBSVIE}. Differently from \cite{CTBSVIE}, we require $(y,z)\in\mathbb{R}\times\mathbb{R}^n$ instead of $(y,z)\in\mathbb{R}^n\times\mathbb{R}^{n\times m}$, which allows us to weaken the hypotheses regarding the differentiability of $g$.

\begin{proposition}[Comparison Theorem]\label{CTBSVIE}
For $i=1,2$, let $g_i:\Omega\times\Delta\times\mathbb{R}^n\to\mathbb{R}$ be independent of $y$, satisfy the assumptions of Theorem \ref{LinfBSVIE} and
$$g_1(t,s,z)\leq g_2(t,s,z) \mbox{ a.s., a.e. } (t,s,z)\in\Delta\times\mathbb{R}^n.$$
Assume also that $\varphi_i\in L^{\infty}((0,T),\mathcal{F}_T)$ have continuous paths and satisfy $\varphi_1(t)\leq\varphi_2(t)$, a.s., a.e. $t \in [0,T]$. Then the solutions $(Y_i(\cdot),Z_i(\cdot,\cdot))$ of the BSVIEs \eqref{rBSVIE} with parameters $(g_i,\varphi_i)$ satisfy:
$$Y_1(t)\leq Y_2(t), \mbox{ a.s., a.e. } t\in[0,T].$$
\end{proposition}

\begin{proof}
This proof is similar to that of Proposition 3.3 in \cite{CTBSVIE}.

Fixing $t\in[0,T]$, let us consider again the families of BSDEs parameterized by $t$:
\begin{equation*}
    \eta^i(r;t)=\varphi^i(t)+\int_r^T g(t,s,\zeta^i(s;t))ds-\int_r^T\zeta^i(s;t)dB_s, \ \ r\in[t,T],
\end{equation*}
for $i=1,2$. Since $g^i(t,s,z)$ and $\varphi^i(t)$ satisfy all the hypotheses of Theorem 2.2 of \cite{EPQ}, 
\begin{equation}
    \eta^1(r;t)\leq \eta^2(r,t), \quad \mbox{ for any } r\in[t,T], \mbox{ a.s.}
    \label{eta}
    \end{equation}
Set now
$$Y^i(t)=\eta^i(t;t) \mbox{ and } Z^i(r;t)=\zeta^i(r;t), \quad \mbox{ for any } (t,r)\in\Delta,$$
 which solve the corresponding BSVIE \eqref{rBSVIE}.
Letting $r\to t^+$ in equation \eqref{eta} and using the pathwise continuity of $\eta(\cdot;t)$ (see Theorem \ref{EUBSDE}), we obtain $Y^1(t)\leq Y^2(t)$ a.s., $t\in[0,T]$.
\end{proof}

\section{Capital allocation rules for cash-subadditive risk measures}  \label{sec: CAR and csa}

This section contains our main results on capital allocation rules induced by cash-subadditive risk measures where the approach of \cite{MastroRos} is extended to the non cash-additive case. First of all, we prove the one-to-one correspondence between cash-subadditive risk measures and capital allocation rules satisfying suitable properties. Afterwards, we consider the setting where risk measures are induced by BSDEs.

Unless otherwise stated, in the sequel we will focus on dynamic risk measures (and capital allocations) defined on $L^{\infty}(\mathcal{F}_T)$.

\subsection{Axiomatic approach}

Let $\rho_t$ be a dynamic risk measure as in \eqref{representationSRM} where we assume that for each $X\in L^{\infty}(\mathcal{F}_T)$ the $\esssup$ is attained for some $({D}^X_t,\mathbb{Q}^X_t)\in\mathcal{D}\times\mathcal{Q}$ (not necessarily unique). Define now
\begin{equation}
    \Lambda^{sub}_t(X,Y)\triangleq\mathbb{E}_{\mathbb{Q}_t^Y}\left[-D^Y_tX\Big|\mathcal{F}_t\right]-c_t(D^Y_t\mathbb{Q}^{Y}_t), \quad X,Y \in L^{\infty}(\mathcal{F}_T),
    \label{subCAR}
\end{equation}
where $$(D^Y_t,\mathbb{Q}^Y_t)\in\argessmax_{(D_t,\mathbb{Q}_t)\in\mathcal{D}\times\mathcal{Q}}\left\{\mathbb{E}_{\mathbb{Q}}\left[-D_tY\Big|\mathcal{F}_t\right]-c_t(D_t\mathbb{Q}_t)\right\}.$$
It is easy to check that the family $(\Lambda_t^{sub})_{t\in[0,T]}$ is a dynamic CAR that henceforth will be called \textit{dynamic cash-subadditive subdifferential} CAR associated to $\rho_t$.
To be more precise, $(\Lambda_t^{sub})_{t\in[0,T]}$ is a family of CARs defined via optimal scenarios. So, when we refer to $(\Lambda_t^{sub})_{t\in[0,T]}$ we mean one of the possible CARs corresponding to a specific scenario $(D_t,\mathbb{Q}_t)\in\mathcal{D}\times\mathcal{Q}$.

Note that the reason why $\Lambda_t^{sub}$ is called a \textit{subdifferential} CAR is because, under suitable hypotheses (see Proposition \ref{prop: optimal} below), a scenario $(D_t,\mathbb{Q}_t)$  is optimal in the representation of $\rho_t(X)$ if and only if $D_t \cdot \mathbb{Q}_t$ belongs to the subdifferential of $\rho_t$ at $X$, defined as
\begin{equation} \label{eq: subdiff-eq}
\partial\rho_t(X)\triangleq\{Z\in L^1_+(\mathcal{F}_T): \ \rho_t(Y)\geq \rho_t(X)+\mathbb{E}[-Z(Y-X)|\mathcal{F}_t] \mbox{ for every } Y\in L^{\infty}(\mathcal{F}_T)\}.
\end{equation}
 As a consequence, more than one optimal scenario is possible when $\rho_t(X)$ is only subdifferentiable, while if $\rho_t(X)$ is also Gateaux differentiable then the subdifferential is a singleton and there exists a unique $D_t \cdot \mathbb{Q}_t$ with $(D_t,\mathbb{Q}_t)$ being an optimal scenario. We recall (see, e.g., \cite{zali}) that the directional derivative of $\rho_t$ at $X$ in the direction $Y$ is defined (if there exists) as
 $$
 D_{\rho}(Y,X)\triangleq\lim_{h \to 0}\frac{\rho_t(X+hY)-\rho_t(X)}{h},
 $$
 while $\rho_t(X)$ is said to be Gateaux differentiable at $X$ if
$$D_{\rho}^+(Y,X)\triangleq\lim_{h \downarrow 0}\frac{\rho_t(X+hY)-\rho_t(X)}{h}$$
exists for any $Y \in L^{\infty}(\mathcal{F}_T)$ and there exists $\nabla \rho_t(X) \in L^1_+(\mathcal{F}_T)$ such that $D_{\rho}^+(Y,X)= \mathbb{E}_{\mathbb{P}}[-Y \nabla \rho_t(X)]$ for any $Y \in L^{\infty}(\mathcal{F}_T)$. In such a case, $\nabla \rho_t(X)$ is unique and is called the Gateaux derivative. Furthermore, also $\nabla \rho_t (Y,X)\triangleq\lim_{h \to 0}\frac{\rho_t(X+hY)-\rho_t(X)}{h}$ exists and $\partial \rho_t(X)= \{ \nabla \rho_t(X) \}$. Note that the definition of Gateaux differentiability is not universal.
\medskip

\begin{proposition} \label{prop: optimal}
Let $\rho_t:L^{\infty}(\mathcal{F}_T)\to L^{\infty}(\mathcal{F}_t)$ be a dynamic cash-subadditive and continuous from above risk measure and suppose that the $\esssup$ in \eqref{representationSRM} is attained for a fixed $X\in L^{\infty}(\mathcal{F}_T)$.

 A scenario $(D^X_t,\mathbb{Q}^X_t)\in\mathcal{D}\times\mathcal{Q}$ is optimal if and only if the density $Z_t^X\triangleq D_t^X \mathbb{Q}_t^X$ is an element of $\partial\rho_t(X)$. Moreover, if there exists $D_{\rho} (Y,X)$ at every $Y \in L^{\infty}(\mathcal{F}_T)$ (in particular, $\rho_t$ is Gateaux differentiable at $X$) then $\partial\rho_t(X)$ is a singleton.
\label{SD}
\end{proposition}

\begin{proof}

Let $X \in L^{\infty}(\mathcal{F}_T)$ be fixed and assume that $(D_t^X,\mathbb{Q}_t^X)$ is an optimal scenario in \eqref{representationSRM}. Then the associated $\lambda_t^X\triangleq D^X_t\mathbb{Q}^X_t$ is an absolutely continuous sub-probability with respect to $\mathbb{P}$. It then follows that, for any $Y\in L^{\infty}(\mathcal{F}_T)$,
\begin{eqnarray}
    \rho_t(Y)-\rho_t(X)&=&\esssup_{\lambda_t\in\widetilde{\mathcal{S}}}\{\mathbb{E}_{\lambda_t}[-Y\big|\mathcal{F}_t]-c_t(\lambda_t)\}-\left(\mathbb{E}_{\lambda_t^X}[-X\big|\mathcal{F}_t]-c_t(\lambda_t^X)\right) \notag \\
    &\geq& \mathbb{E}_{\lambda_t^X}[-(Y-X)\big|\mathcal{F}_t] , \label{eq: star0000}
\end{eqnarray}
where $\widetilde{\mathcal{S}}\triangleq\left\{\lambda_t \mbox{ sub-probability}: \lambda_t=D_t\mathbb{Q}_t \mbox{ for some } (D_t,\mathbb{Q}_t)\in \mathcal{D}\times\mathcal{Q}\right\}$.
Hence, the density of  the sub-probability $\lambda_t^X$ defined by $Z^X_t\triangleq\frac{d\lambda_t^X}{d\mathbb{P}}=D_t^X\frac{d\mathbb{Q}^X_t}{d\mathbb{P}}$ is an element of the subdifferential $\partial\rho_t(X)$.

Conversely, we are now going to prove that any element of the subdifferential can be identified with an optimal scenario $(D^X_t,\mathbb{Q}^X_t)$. We first prove that, given $X \in L^{\infty}(\mathcal{F}_T)$, if $Z\in\partial\rho_t(X)$ then its $L^1$-norm is less or equal to $1$. Indeed, by taking $Y=X-m$ with $m\in\mathbb{R}_+$ and using cash-subadditivity, it follows that
$$\rho_t(X)+m\geq\rho_t(Y)\geq \rho_t(X)+m \, \mathbb{E}[Z\big|\mathcal{F}_t],$$
where the last inequality is due to $Z \in \partial \rho_t(X)$. This implies that $0\leq\mathbb{E}[Z\big|\mathcal{F}_t]\leq 1$, hence that $0\leq\mathbb{E}[Z]\leq 1$ and $Z$ is a density corresponding to a sub-probability $\lambda$ through $Z\triangleq\frac{d\lambda}{d\mathbb{P}}$. From now on, we identify the sub-probability $\lambda$ with its density, so any element of $\partial\rho_t(X)$ is understood as a sub-probability via this identification. By Proposition 6 in \cite{MastroRos1}, we know that every sub-probability $\lambda_t$ absolutely continuous w.r.t. $\mathbb{P}$ such that $c_{0,t}(\lambda_t)<+\infty$ can be decomposed as $\lambda_t=D_t\mathbb{Q}_t$ with $(D_t,\mathbb{Q}_t)\in\mathcal{D}\times\mathcal{Q}$, where $c_{0,t}$ is the minimal penalty function of the static cash-subadditive (continuous from above) risk measure $\rho_{0,t}(X)\triangleq\mathbb{E}[\rho_t(X)]$ (see Proposition 6 in \cite{MastroRos1}). If $\lambda_t^X\in\partial\rho_t(X)$ (meaning that $Z_t^X \in \partial\rho_t(X)$) then, by applying the $\mathbb{P}$-expectation to the inequality in \eqref{eq: star0000} and by using the fact that $\mathbb{E}_{\mathbb{P}}[\mathbb{E}_{\lambda^X_t}[\cdot\big|\mathcal{F}_t]]=\mathbb{E}_{\lambda_t^X}[\cdot]$,
\begin{equation} \label{eq: ineq-rho}
\mathbb{E}_{\lambda_t^X}[-Y]-\rho_{0,t}(Y)\leq \mathbb{E}_{\lambda_t^X}[-X]-\rho_{0,t}(X) \ \mbox{for every } Y \in L^{\infty}(\mathcal{F}_t).
\end{equation}
By the representation for the penalty term of a static cash-subadditive risk measure (see \cite{ElRa}) and by \eqref{eq: ineq-rho}, we have:
$$c_{0,t}(\lambda_t^X)=\sup_{Y\in L^{\infty}(\mathcal{F}_T)} \left\{\mathbb{E}_{\lambda_t^X}[-Y]-\rho_{0,t}(Y)\right\}=\mathbb{E}_{\lambda_t^X}[-X]-\rho_{0,t}(X) <+\infty.$$
So far, we have proved that $\lambda_t^X$ can be decomposed as  $\lambda_t^X=D_t^X\mathbb{Q}^X_t$. Now, by repeating the same arguments as above by replacing $\rho_{0,t}$ with $\rho_t$, it holds that
$$c_t(D_t^X\mathbb{Q}^X_t)=\mathbb{E}_{\mathbb{Q}^X_t}[-D_t^X X\big|\mathcal{F}_t]-\rho_t(X)$$
and, consequently, that $(D^X_t,\mathbb{Q}^X_t) \in \mathcal{D} \times \mathcal{Q}$ is an optimal scenario.

Finally, it remains to prove that, under the existence of the directional derivative at every direction, $\partial \rho_t(X)$ is a singleton. Assume that there exists $D_{\rho} (Y,X)$ at every $Y \in L^{\infty}(\mathcal{F}_T)$. By proceeding similarly to Remark 3 in \cite{MastroRos}, we have
\begin{equation} \label{eq: right-limit}
\lim_{h\downarrow 0}\frac{\rho_t(X+hY)-\rho_t(X)}{h}\geq \lim_{h\downarrow 0} \frac{\mathbb{E}_{\lambda_t^X}[-hY\big|\mathcal{F}_t]}{h}=\mathbb{E}_{\lambda_t^X}[-Y\big|\mathcal{F}_t],
\end{equation}
and, analogously, for the left-limit:
\begin{equation} \label{eq: left-limit}
\lim_{h \uparrow 0}\frac{\rho_t(X+hY)-\rho_t(X)}{h}\leq \lim_{h \uparrow 0} \frac{\mathbb{E}_{\lambda_t^X}[-hY\big|\mathcal{F}_t]}{h}=\mathbb{E}_{\lambda_t^X}[-Y\big|\mathcal{F}_t],
\end{equation}
where inequalities follow from the dual representation of the risk measure $\rho_t$ and from the assumption that $\lambda_t^X=D^X_t\mathbb{Q}^X_t$, with $(D^X_t,\mathbb{Q}^X_t)\in \mathcal{D}\times\mathcal{Q}$, is an optimal scenario. Existence of the directional derivative at every direction, \eqref{eq: right-limit} and \eqref{eq: left-limit} then imply that $$\nabla\rho_t(X;Y)=\lim_{h\to0}\frac{\rho_t(X+hY)-\rho_t(X)}{h}=\mathbb{E}_{\lambda_t^X}[-Y\big|\mathcal{F}_t].$$
Suppose now that there exist two optimal scenarios $\lambda_{t,1}^X,\lambda_{t,2}^X \in \partial \rho_t (X)$. In this case, it would then hold that $$\nabla\rho_t(X;Y)=\mathbb{E}_{\lambda_{t,1}^X}[-Y\big|\mathcal{F}_t]=\mathbb{E}_{\lambda_{2,t}^X}[-Y\big|\mathcal{F}_t] \quad \mbox{ for any } Y\in L^{\infty}(\mathcal{F}_T),$$ which implies that the two densities coincide almost surely, hence that $\partial\rho_t(X)$ is a singleton.
\end{proof}

Similarly as proved in \cite{MastroRos} for the cash-additive case, the following correspondence between risk measures and CARs holds also in the cash-subadditive case.

\begin{proposition}
a) If $\Lambda_t:L^{\infty}(\mathcal{F}_T)\times L^{\infty}(\mathcal{F}_T)\to L^{\infty}(\mathcal{F}_t)$, with $t\in[0,T]$, is a dynamic CAR satisfying no-undercut, monotonicity, weak-convexity, normalization, 1-time-consistent (resp. time-consistent), and cash-subadditivity, then $\rho_t(X)\triangleq\Lambda_t(X,X)$ is a weak time-consistent (resp. time-consistent) convex cash-subadditive risk measure. \smallskip

\noindent b) Conversely, if $\rho_t$ is a dynamic convex cash-subadditive risk measures that can be represented as in \eqref{representationSRM} with the $\esssup$ attained, then there exists at least a no-undercut, monotone, weakly-convex and 1-cash-subadditive dynamic CAR $(\Lambda_t)_{t\in[0,T]}$ which also satisfies sub-allocation. In particular, $\Lambda_t^{sub}$ is a possible CAR of this kind.
\label{Subproperties}
\end{proposition}

\begin{proof}
a) The proof of normalization, convexity, monotonicity, and weak time-consistency (resp. time-consistency) follows immediately from Proposition 5 in \cite{MastroRos}. The cash-subadditivity of $\rho_t$ is easily given by
$$\rho_t(X+m_t)=\Lambda_t(X+m_t,X+m_t)\geq \Lambda_t(X,X)-m_t=\rho_t(X)-m_t$$
true for any $X \in L^{\infty}(\mathcal{F}_T), m_t\in L^{\infty}_+(\mathcal{F}_t)$,
where the inequality is due to cash-subadditivity of $\Lambda_t$.\smallskip

b) We are now going to prove the existence of such a $\Lambda_t$ by checking that, e.g., the subdifferential CAR $\Lambda_t^{sub}$ defined in \eqref{subCAR} satisfies all the properties required. First of all, $\Lambda_t^{sub}(X,X)=\rho_t(X)$ for any $X\in L^{\infty}(\mathcal{F}_T),t\in[0,T]$, by definition of $\Lambda_t^{sub}$. Monotonicity, weakly-convexity, sub-allocation, and no-undercut follow similarly as in Proposition 3 of \cite{MastroRos}. Concerning 1-cash-additivity, it holds that, for any $m_t\in L^{\infty}_+(\mathcal{F}_t), X,Y \in L^{\infty}(\mathcal{F}_T)$,
\begin{eqnarray*}
\Lambda_t^{sub}(X+m_t,Y)&=&\mathbb{E}_{\mathbb{Q}_t^Y}\left[-D^Y_t(X+m_t)\Big|\mathcal{F}_t\right]-c_t(D^Y_t\mathbb{Q}^Y_t) \\
&=&\mathbb{E}_{\mathbb{Q}_t^Y}\left[-D^Y_tX\Big|\mathcal{F}_t\right]-m_t \, \mathbb{E}_{\mathbb{Q}_t^{Y}}[D_t Y|\mathcal{F}_t]-c_t(D^Y_t\mathbb{Q}^Y_t) \\
&\geq& \mathbb{E}_{\mathbb{Q}_t^Y}\left[-D^Y_tX\Big|\mathcal{F}_t\right]-m_t-c_t(D^Y_t\mathbb{Q}^Y_t) \\
&=& \Lambda_t^{sub}(X,Y)-m_t,
\end{eqnarray*}
where the first equality is due to the $\mathcal{F}_t$-measurability of $m_t$, while the inequality follows from $0\leq D_t\leq 1$. The proof is therefore complete.
\end{proof}

Note that while time-consistency of $\Lambda_t$ implies that of $\rho_t$, the converse implication does not necessarily hold true. In the particular case of cash-additive risk measures induced by BSDEs, however, the corresponding CARs are time-consistency, as shown in \cite{MastroRos}.

\subsection{Capital allocation rules and backward equations} \label{sec: cars and bsvies}

We focus now our attention on cash-subadditive risk measures induced by BSDEs. Following El Karoui and Ravanelli \cite{ElRa}, then, we consider a BSDE of the form:
\begin{equation}
\rho_t(X)=-X+\int_t^Tg(s,\rho_s(X),Z^{\rho}_s)ds-\int_t^TZ^{\rho}_s \,dB_s, \quad X \in L^{\infty} (\mathcal{F}_T),
\label{CSDRM}
\end{equation}
where the driver $g(s,y,z) $ is convex in $(y,z)$, decreasing in $y$ and satisfies the standard assumptions. Throughout this section we require that, in the dual representation of $\rho_t$, the optimal scenarios $(D_t,\mathbb{Q}_t)$ belong to $\mathcal{D}'\times\mathcal{Q}$, without the further restriction $D_t\in\mathcal{D}$.\smallskip

Inspired by what done in \cite{MastroRos} for the cash-additive case (where $g$ does not depend on $y$), we investigate now the case of dynamic subdifferential CARs induced by these cash-subadditive risk measures. Differently from the cash-additive case where BSDEs (or BSVIEs) were preserved when moving from risk measures to capital allocations, we show that this happens no more in the cash-subadditive case. For simplicity, we always work with one of the possible optimal scenarios, in order to restrict the family of $(\Lambda_t^{sub})_{t\in[0,T]}$ to one of its element.

We start with the following motivating example.

\begin{example}\label{subexample}
Following the Example 7.2 in \cite{ElRa}, we consider the driver
$$g_{\beta}(t,y,z)=\sup_{r_t\leq \beta_t\leq R_t}(-\beta_t y)=-\bar{\beta}_t y,$$
where $(\beta_t)_{t\in[0,T]}$ is a given ambiguous discount rate with $0 \leq r_t \leq \beta_t \leq R_t$ with $(r_t)_{t\in[0,T]}$ and $(R_t)_{t\in[0,T]}$ being two adapted and bounded processes. It is well known that the BSDE \eqref{CSDRM} with driver $g_{\beta}$ admits as a unique solution:
\begin{equation} \label{eq: rho-ambig}
\rho_t(Y)=\mathbb{E}_\mathbb{P}\left[\left.e^{-\int_t^T\bar{\beta}_sds}(-Y)\right|\mathcal{F}_t\right].
\end{equation}
It then follows that $(\mathbb{P},D^Y_t\triangleq e^{-\int_t^T\bar{\beta}_sds})$ is an optimal scenario for $\rho_t(Y)$ with $c_t(D_t \mathbb{P})=0$ and, by definition,
$$\Lambda^{sub}_t(X,Y)=\mathbb{E}_\mathbb{P}[\left.-D_t^Y X\right|\mathcal{F}_t].$$
Note that here $D^Y_t$ is not necessarily adapted, but only satisfies $D^Y_t\in [0,1]$.
\smallskip

By using the Martingale Representation Theorem, then, for any fixed $t \in [0,T]$ there exists a unique predictable process $Z^{\Lambda^{sub}}(t,\cdot)=Z^{\Lambda^{s}}(t,\cdot)$ defined on $[t,T]$ such that:
\begin{equation} \label{eq: lambda-sub}
\Lambda^{sub}_t(X,Y)=-D^Y_tX-\int_t^T Z^{\Lambda^{s}}(t,s)ds.
\end{equation}
Surprisingly, starting from a cash-subadditive risk measure $\rho_t(Y)$ induced by a BSDE, the dynamics of $\Lambda_t^{sub}(X,Y)$ can be only written in terms of a BSVIE which, by Theorem \ref{EUBSVIE}, admits a unique solution. This can be explained by the fact that, because of the discounting factor $D^Y_t$, also the final condition in \eqref{eq: lambda-sub} depends on $t$.
\end{example}

The previous example underlines how, when dropping cash-additivity, it may be possible to obtain a CAR that does not follow a classical BSDE, even when the associated risk measure is induced by a BSDE. This is different from the cash-additive case and explains why we investigate which kind of backward equation is followed by a subdifferential CAR associated to a cash-subadditive risk measure induced by a BSDE.

The impossibility of obtaining $\Lambda_t(X,Y)$ through a BSDE suggests a lack of time-consistency for CARs induced via cash-subadditive risk measures: this is not necessarily a drawback. Indeed, as argued in Yong \cite{Yong2006,Yong2007} and the references therein, time-inconsistent preferences exist in real markets, so the time-inconsistency of CARs driven by BSVIEs can be used to model investors with such preferences. Moreover, the impossibility to reduce the dynamics of CARs to a BSDE is not completely new: a similar phenomenon, indeed, happens in the case of standard expected utility (see \cite{Nacira,Yong2006,Yong2007}).
\medskip

We are now ready to provide the dynamics of the subdifferential CAR generated by a cash-subadditive risk measure with a generic driver $g$.

\begin{proposition} \label{BSVIEsub}
Let $\rho_t: L^{\infty}(\mathcal{F}_T) \to L^{\infty}(\mathcal{F}_t)$ be a dynamic cash-subadditive convex risk measure induced by a BSDE as in \eqref{CSDRM} whose the driver $g$ is convex in $(y,z)$, decreasing in $y$ and satisfies the standard assumptions. Then, for any $X,Y\in L^{\infty}(\mathcal{F}_T)$, the subdifferential CAR defined as in \eqref{subCAR} is the first component of the solution to the following BSVIE that admits a unique solution $(\Lambda^{sub},Z^{\Lambda^s})\in L^{\infty}((0,T),\mathbb{F})\times L^{\infty}((0,T);\mbox{BMO}(\mathbb{P}))$:
\begin{equation}
    \hspace{-0.6cm}\Lambda^{sub}_t(X,Y) = -D_tX    +\int_t^T \hspace{-0.2cm}g_{\Lambda^{s}}(t,s,\rho_s(Y),Z_s^{\rho(Y)},Z^{\Lambda^s}(t,s)) ds-\int_t^TZ^{\Lambda^{s}}(t,s) \,dB_s,  \label{eq: lambda-sub2}
\end{equation}
where $g_{\Lambda^s}:\Omega\times\Delta\times\mathbb{R}\times\mathbb{R}^n\times\mathbb{R}^n\to\mathbb{R}$ is defined by:
\begin{equation}
\hspace{-0.2cm}g_{\Lambda^{s}}(t,s,\rho_s(Y),Z_s^{\rho(Y)},Z^{\Lambda^s}(t,s)) \triangleq \hspace{-0.1cm}D_{t,s}\left[-g(s,\rho_s(Y),Z^{\rho(Y)}_s)-\beta^Y_s\rho_s(Y)\hspace{-0.1cm}-\mu_s^Y Z^{\rho}_s\right] \hspace{-0.1cm}+ \mu^Y_s Z^{\Lambda^{s}}(t,s),
	\label{subdriver}
\end{equation}
where $D_t\triangleq\exp\left\{-\int_t^T\beta^Y_s ds \right\}$ and $D_{t,s}\triangleq \exp\left\{-\int_t^s\beta^Y_udu \right\}$ with
$(\beta^Y_t,\mu^Y_t)$ being (one of) the optimal scenario in the representation \eqref{BSDEmax} of $\rho_t(Y)$, identifying $\mu_t^Y$ with the corresponding $\mathbb{Q}_t^{\mu^Y}$.

\noindent Furthermore, $\Lambda_t^{sub}(X,Y)$ satisfies no-undercut, monotonicity, 1-cash-subadditivity, weakly-convexity, and sub-allocation.
Moreover, the sub-probability $\lambda_t^Y\triangleq e^{-\int_t^T{\beta}^Y_sds}\, \mathbb{Q}^{{\mu}^Y}_t$ belongs to $\partial\rho_t(Y)$.
\end{proposition}

\begin{proof}
First of all, the BSVIE \eqref{eq: lambda-sub2} admits a unique solution in $L^{\infty}((0,T),\mathbb{F})\times L^{\infty}((0,T);\mbox{BMO}(\mathbb{P}))$ because the driver $g_{\Lambda^s}$ satisfies all the assumptions of Theorem \ref{LinfBSVIE}. Indeed, the $(\mathcal{F}_t)_{t\in[0,T]}$-progressive measurability follows from the same property of $g(t,y,z)$ and $(\beta^Y_t,\mu^Y_t)$, while the uniform Lipschitzianity w.r.t. $Z^{\Lambda}$ follows from the boundedness of $\mu^Y_t$. Moreover, we have:
 \begin{equation*}
 	|g^{\Lambda^s}(t,s,\rho_s^Y,Z^{\rho(Y)}_s)|^2\leq C(|g(s,0,0)|^2+|Z_s^{\rho(Y)}|^2+1),
 \end{equation*}
where we used $D_{t,s}\in L^{\infty}((0,T),\mathcal{F}_s)$, Lipschitzianity of $g$ and boundedness of $(\beta_t^Y,\mu^Y_t)$ and $\rho_t(Y)$. By integrating and taking the expectation w.r.t. $\mathcal{F}_r$ and the supremum over $r\in[t,T]$,
 \begin{equation*}
\hspace{-0.1cm}	\sup_{r\in[t,T]}\mathbb{E}\left[\left. \int_r^T|g^{\Lambda^s}(t,s,\rho_s^Y,Z^{\rho(Y)}_s)|^2ds \right|\mathcal{F}_r\right]\leq C_1 \hspace{-0.1cm} \sup_{r\in[t,T]}\mathbb{E}\left[\left. \int_r^T\left(|g(s,0,0)|^2+|Z_s^{\rho(Y)}|^2+1\right)ds \right|\mathcal{F}_r\right]\hspace{-0.1cm} \leq \hspace{-0.1cm} C_2
\end{equation*}
for some $C_1,C_2>0$, where the last inequality follows from $g(\cdot,0,0),Z_{\cdot}^{\rho(Y)}\in \mbox{BMO}(\mathbb{P})$ (see Proposition \ref{regularityBSDE}). Since $C_1$ and $C_2$ do not depend on $t$ or $\omega$, 
\begin{equation*}
\sup_{t\in[0,T]}\left\|\sup_{r\in[t,T]}\mathbb{E}\left[\left. \int_r^T|g^{\Lambda^s}(t,s,\rho_s^Y,Z^{\rho(Y)}_s)|^2ds \right|\mathcal{F}_r\right]\right\|_{\infty}\leq C_2<+\infty,
\end{equation*}
thus $g_{\Lambda^s}$ satisfies the assumptions of Theorem \ref{LinfBSVIE}.

The rest of the proof is similar to that used in the proof of Proposition 6 in \cite{MastroRos}.
By the dual representation of $\rho_t(Y)$ with optimal scenario $(\beta^Y_t,\mu^Y_t)$,
\begin{equation*}
    \rho_t(Y)=\mathbb{E}_{\mathbb{Q}^{\mu^Y}}\left[\left. e^{-\int_t^T\beta^Y_s ds} \, (-Y)-\int_t^Te^{-\int_t^s\beta^Y_udu} \, \, G(s,\beta^Y_s,\mu^Y_s)ds \right|\mathcal{F}_t\right].
\end{equation*}
By definition, then,
\begin{equation}
    \Lambda^{sub}_t(X,Y)=\mathbb{E}_{\mathbb{Q}^{\mu^Y}}\left[\left. e^{-\int_t^T\beta^Y_s ds} \, (-X)-\int_t^Te^{-\int_t^s\beta^Y_udu} \, \, G(s,\beta^Y_s,\mu^Y_s)ds \right|\mathcal{F}_t\right],
    \label{subequation}
\end{equation}
where the explicit expression of the penalty term appears.
By the Martingale Representation Theorem, for any fixed $t \in [0,T]$ there exists a predictable process $Z^{\Lambda^s}(t,\cdot)$ such that:
$$\Lambda^{sub}_t(X,Y)=e^{-\int_t^T\beta^Y_s ds}(-X)-\int_t^Te^{-\int_t^s\beta^Y_udu}G(s,\beta^Y_s,\mu^Y_s) \, ds-\int_t^TZ^{\Lambda^s}(t,s) \, dB^{\mu^Y}_s,$$
 where $dB_t^{\mu^Y}=dB_t+\mu^Y_tdt$. Using Lemma 7.4 (ii) of \cite{ElRa} and making the definition of $B_t^{\mu^Y}$ explicit, the first statement of the result follows.

 By construction, $\Lambda_t^{sub}$ is a CAR since  $\Lambda^{sub}_t(Y,Y)=\rho_t(Y)$ for any $Y\in L^{\infty}(\mathcal{F}_T)$. No-undercut, monotonicity, 1-cash-subadditivity, weakly-convexity, and sub-allocation are consequences of Proposition \ref{Subproperties}.

It remains to prove that $\lambda_t^Y \in \partial\rho_t(Y)$. To this aim, we remind (see \cite{ElRa,zali}) that an element $(\beta,\mu)\in\mathbb{R}\times\mathbb{R}^n$ attains the sup in the Fenchel transformation of $g$ (i.e. $G(t,\beta,\mu)=-g(t,y,z)-\beta y-\mu\cdot z$) if and only if $(\beta,\mu)\in\partial g(t,y,z)\triangleq\{(\bar{\beta},\bar{\mu})\in\mathbb{R}\times\mathbb{R}^n: g(t,v,w)\geq g(t,y,z)-\bar{\beta}(v-y)-\bar{\mu}(w-z) \, \mbox{ for every } (v,w)\in\mathbb{R}\times\mathbb{R}^n\}$. Fix now $X\in L^{\infty}(\mathcal{F}_T)$. For every $Y\in L^{\infty}(\mathcal{F}_T)$ it then holds that
\begin{eqnarray*}
-d(\rho_t(Y)-\rho_t(X))&=&\left[g(t,\rho_t(Y),Z^Y_t)-g(t,\rho_t(X),Z^X_t)\right]dt- (Z^Y_t-Z_t^X) dB_t \\
&\geq & -\beta^X_t(\rho_t(Y)-\rho_t(X))dt-\mu^X_t (Z^Y_t-Z_t^X) dt -(Z^Y_t-Z_t^X)dB_t \\
&=& -\beta^X_t(\rho_t(Y)-\rho_t(X))dt- (Z^Y_t-Z_t^X) dB_t^{\mu^X}
\end{eqnarray*}
with final condition $-(Y-X)$, where $dB_t^{\mu^X}\triangleq dB_t+\mu^X_tdt$ and where the inequality is due to the definition of subdifferential of $g$ and to Theorem 7.1 in \cite{ElRa}. Consider now the following BSDE:
$$-d\xi_t=-\beta^X_t\xi_t \, dt -\Gamma_t \, dB^{\mu^X}_t \mbox{ and } \xi_T=-(Y-X),$$
with solution
$$\xi_t=\mathbb{E}_{\mathbb{Q}_t^{\mu^X}}\left[-e^{-\int_t^T\beta_sds} \, (Y-X)\Big|\mathcal{F}_t\right]=\mathbb{E}_{\lambda_t^X}\left[-(Y-X)\Big|\mathcal{F}_t\right],$$
where $\lambda_t^X\triangleq e^{-\int_t^T{\beta}^X_sds} \, \mathbb{Q}^{{\mu}^X}_t$.

By applying the Comparison Theorem for BSDEs (see Theorem 7.1 in \cite{ElRa}) and by the arguments above, we have:
\begin{equation*}
\rho_t(Y)-\rho_t(X)\geq \xi_t=\mathbb{E}_{\lambda_t^X}\left[-(Y-X)\Big|\mathcal{F}_t\right], \quad \mbox{ a.s. } \forall t\in[0,T].
\end{equation*}
Hence $\lambda_t^X$ belongs to $\partial\rho_t(X)$.
\end{proof}
\medskip

In the previous result we have shown that,
given a cash-subadditive risk measure $\rho_t$ induced by a BSDE, there exists a natural way to define $\Lambda_t^{sub}$ as solution to a BSVIE. Once fixed an optimal scenario $(D_t^{\beta},\mathbb{Q}_t^{\mu})$, the one-to-one correspondence between $\rho_t$ and $\Lambda^{sub}_t$ suggests that we can build several different CARs induced by BSVIEs considering the dynamics of the subdifferential CAR instead of using the underlying risk measure.
More in general, following the strategy adopted in \cite{MastroRos}, we can define and build a CAR $\Lambda_t$ as the first component of the solution to the BSVIE:
\begin{equation}
    \Lambda_t(X,Y)=-D_tX+\int_t^Tg_{\Lambda}(t,s,\rho_s(Y),Z_s^{\rho_s(Y)},Z^{\Lambda^s(Y,Y)}(t,s),Z^{\Lambda}(t,s))ds-\int_t^TZ^{\Lambda}(t,s)dB_s,
    \label{CARBSVIE}
\end{equation}
where the driver $g_{\Lambda}(\cdot ,\cdot ,\rho,Z^{\rho},Z^{\Lambda^s},\cdot)$ satisfies the assumptions in Theorem \ref{LinfBSVIE} plus the further condition
\begin{equation}
g_{\Lambda}(t,s,\rho,Z^{\rho},Z^{\Lambda^s},Z^{\Lambda^s})=g_{\Lambda^s}(t,s,\rho,Z^{\rho},Z^{\Lambda^s}) \quad \mbox{ for any } (t,s)\in\Delta, \rho\in\mathbb{R},Z^{\rho},Z^{\Lambda^s}\in\mathbb{R}^n.
\label{HPg}
\end{equation}
We note that the hypothesis \eqref{HPg} guarantees that $\Lambda_t(X,X)=\Lambda^{sub}_t(X,X)=\rho_t(X)$ for every $X\in  L^{\infty}(\mathcal{F}_T)$ and $t\in[0,T]$, implying that $\Lambda_t$ is a CAR. Indeed, according to the uniqueness of the solution for a BSVIE and taking $(\Lambda_t(X,X),Z^{\Lambda}(t,s))=(\Lambda^{sub}_t(X,X),Z^{\Lambda^{s}}(t,s))$, it holds that:
\begin{align*}
\Lambda_t(X,X)&=-D_tX+\int_t^Tg_{\Lambda}(t,s,\rho_s(X),Z_s^{\rho_s(X)},Z^{\Lambda^s(X,X)}(t,s),Z^{\Lambda}(t,s))ds-\int_t^TZ^{\Lambda}(t,s)dB_s \\
&=-D_tX+\int_t^Tg_{\Lambda^s}(t,s,\rho_s(X),Z_s^{\rho_s(X)},Z^{\Lambda^s(X,X)}(t,s))ds-\int_t^TZ^{\Lambda}(t,s)dB_s \\
&=\Lambda_t^{sub}(X,X).
\end{align*}
We stress that, while $\Lambda_t(X,X)$, $\Lambda_t^{sub}(X,X)$ and $\rho_t(X)$ coincide a.s. $\forall t\in[0,T]$, the same does not necessarily holds for $Z^{\Lambda^s(X,X)}\equiv Z^{\Lambda(X,X)}$ and $Z^{\rho(X)}$.

We provide here below some sufficient conditions on the driver $g_{\Lambda}$ under which the induced CAR satisfies some further properties.

\begin{proposition}
Let $g_{\Lambda}:\Omega\times\Delta\times\mathbb{R}\times\mathbb{R}^{n}\times\mathbb{R}^{n}\times\mathbb{R}^{n}\to\mathbb{R}$ be a driver satisfying the assumptions in Theorem \ref{LinfBSVIE} and condition \eqref{HPg} and let $\Lambda$ be defined as in \eqref{CARBSVIE}.\medskip

\noindent i) If $X\geq Y$, with $X,Y\in L^{\infty}(\mathcal{F}_T)$, then $\Lambda_t(X,Z)\leq\Lambda_t(Y,Z)$ for any $Z\in L^{\infty}(\mathcal{F}_T)$ and $t\in[0,T]$.\smallskip

\noindent ii) $\Lambda_t(X+m_t,Y)\geq\Lambda_t(X,Y)-m_t$ for any $X,Y\in L^{\infty}(\mathcal{F}_T), m_t\in L^{\infty}_+(\mathcal{F}_t), t\in[0,T]$.\smallskip

\noindent iii) If $g_{\Lambda}(t,s,\rho,Z^{\rho},Z^{\Lambda^s},0)=0$ for any $(t,s)\in\Delta,\rho\in\mathbb{R},Z^{\rho}, Z^{\Lambda^s}\in\mathbb{R}^n$, then $\Lambda_t(0,Y)=0$ for any $t\in[0,T], Y\in L^{\infty}(\mathcal{F}_T)$.\smallskip

\noindent iv) If $g_{\Lambda}(t,s,\rho,Z^{\rho},Z^{\Lambda^s},Z^{\Lambda})\leq g_{\Lambda^s}(t,s,\rho,Z^{\rho},Z^{\Lambda})$ for any $\rho\in\mathbb{R},Z^{\rho},Z^{\Lambda^s},Z^{\Lambda}\in\mathbb{R}^n,(t,s)\in\Delta$, then $\Lambda_t$ satisfies no-undercut.\smallskip

\noindent v) If $\sum_{i=1}^ng_{\Lambda}(t,s,\rho,Z^{\rho},Z^{\Lambda^s},Z^{\Lambda^i})\leq g_{\Lambda}(t,s,\rho,Z^{\rho},Z^{\Lambda^s},\sum_{i=1}^nZ^{\Lambda^i})$ for any $(t,s)\in\Delta, \rho\in\mathbb{R}$, \linebreak $Z^{\rho},Z^{\Lambda^s},Z^{\Lambda^i}\in\mathbb{R}^n$ (for $ i=1,\dots,n$), then $\Lambda_t$ satisfies sub-allocation for any $t\in[0,T]$.\smallskip

\noindent vi) If $g_{\Lambda}$ is convex in $Z^{\Lambda}$ for any $ (t,s)\in\Delta, \rho\in\mathbb{R},, Z^{\rho},Z^{\Lambda^s}\in\mathbb{R}^n$, then $\Lambda_t$ satisfies weak-convexity.
\end{proposition}

\begin{proof}
\noindent i) If $X\geq Y$ then $-D_tX\leq-D_tY$ are elements of $L^{\infty}([0,T],\mathcal{F}_T;\mathbb{R})$ with continuous paths. The thesis then follows by applying the Comparison Theorem (see Proposition \ref{CTBSVIE}).\smallskip

\noindent ii), v), vi) can be checked easily by applying the Comparison Theorem.\smallskip

\noindent iii) follows immediately by $\Lambda_t(0,Y)\equiv0$ and $Z^{\Lambda}(t,s)\equiv0$ and by the uniqueness of the solution.

\noindent iv) For any $X, Y \in L^{\infty}(\mathcal{F}_T)$
 \begin{eqnarray*}
        \Lambda_t(X,Y) &=&-D_tX+\int_t^Tg_{\Lambda}(t,s,\rho_s(Y),Z_s^{\rho(Y)},Z^{\Lambda^s(Y,Y)}(t,s),Z^{\Lambda}(t,s))ds-\int_t^TZ^{\Lambda}(t,s)dB_s \\
        &\leq & -D_tX+\int_t^Tg_{\Lambda^s}(t,s,\rho_s(Y),Z_s^{\rho(Y)},Z^{\Lambda}(t,s))ds-\int_t^TZ^{\Lambda}(t,s)dB_s \\
        &=& \Lambda^{sub}_t(X,Y)\leq\rho_t(X) \ \mbox{ a.s. } \forall t\in[0,T],
    \end{eqnarray*}
    where the first inequality is given by the Comparison Theorem, the second equality is due to the uniqueness of the solution, while the last inequality follows from Proposition \ref{BSVIEsub} (no-undercut of $\Lambda_t^{sub}$).
\end{proof}

We focus now on two relevant capital allocation rules: the gradient and the Aumann-Shapley, generalized to our context.

\subsubsection{Gradient capital allocation}

Let $\rho_t:L^{\infty}(\mathcal{F}_T)\to L^{\infty}(\mathcal{F}_t)$ be a dynamic cash-subadditive convex risk measure induced by the BSDE \eqref{CSDRM}, where the driver $g$ is convex in $(y,z)$, decreasing in $y$, and satisfies the standard assumptions. Assume also that $g(t,y,z)$ is continuously differentiable w.r.t. $(y,z)\in\mathbb{R}\times\mathbb{R}^n$. Moreover, for $X,Y\in L^{\infty}(\mathcal{F}_T)$, the \textit{dynamic gradient CAR} is defined as
  $$\Lambda^{grad}_t(X,Y)\triangleq \lim_{h \to 0} \frac{\rho_t(X + hY)- \rho_t(X)}{h}=\nabla\rho_t(X;Y).$$
   Note that, under the assumptions above, the existence of the directional derivative of $\rho_t$ at $X$ in every direction is guaranteed by Theorem 2.1 in \cite{Ankirchner}.

\begin{proposition}
Under the assumptions above on $\rho_t$ and on $g$, $\Lambda_t^{grad}(X,Y)$ is the first component of the solution to the following BSVIE:
\begin{equation} \label{eq: dyn-grad}
\Lambda^{grad}_t(X,Y)=-Xe^{-\int_t^T\beta_udu}-\int_t^T \mu_s Z^{grad}(t,s) \, ds-\int_t^T Z^{grad}(t,s) \, dB_s,
\end{equation}
where $\beta_t\triangleq-\partial_{\rho}g \left(t,\rho_t(Y),Z^{\rho(Y)}_t \right)$ and $\mu_t=-\nabla_z g\left(t,\rho_t(Y),Z_t^{\rho(Y)}\right)$.
\enlargethispage{+0.5cm}
\label{gradientCAR}
\end{proposition}

\begin{proof}
It is easy to check that all the hypotheses of Theorem 2.1 of \cite{Ankirchner} are fulfilled. Then $\Lambda^{grad}_t(X,Y)$ satisfies
\begin{eqnarray*}
\Lambda^{grad}_t(X,Y)&=&-X-\int_t^T \left(\beta_s\Lambda^{grad}_s(X,Y)- \mu_s Z^{\nabla}_s \right) \, ds-\int_t^T Z^{\nabla}_s \,dB_s \\
&=&-X-\int_t^T\beta_s\Lambda^{grad}_s(X,Y)ds-\int_t^T Z^{\nabla}_s dB^{\mu}_s,
\end{eqnarray*}
where $dB^{\mu}_t=dB_t+\mu_tdt$.
By Proposition 2.2 of \cite{EPQ}, $\Lambda^{grad}$ can be rewritten as
$$\Lambda^{grad}_t(X,Y)=\mathbb{E}_{\mathbb{Q}_t^{\mu}}[e^{-\int_t^T\beta_udu} \, (-X)|\mathcal{F}_t],$$
where $\mathbb{E}_{\mathbb{P}} [\frac{d \mathbb{Q}^{\mu}}{d \mathbb{P}} | \mathcal{F}_t ]=\exp(\mu\cdot B)_t$.
Finally, by the Martingale Representation Theorem, for any fixed $t\in[0,T]$ there exists a predictable process $Z^{\grad}(t,\cdot)$ such that:
$$\Lambda^{grad}_t(X,Y)=-Xe^{-\int_t^T\beta_udu}-\int_t^T Z^{grad}(t,s) \, dB^{\mu}_s,$$ hence the thesis.
\end{proof}
\medskip

Note that the proof of Proposition \ref{gradientCAR} suggests a representation for the gradient CAR $\Lambda_t^{grad}$ as in \eqref{eq: dyn-grad},
which is similar to what we have obtained in Example \ref{subexample}.
It is also worth emphasizing that, if $g$ is decreasing in $y$, sublinear, and continuously differentiable w.r.t. $(y,z)$, then $\Lambda^{sub}_t(X,Y)\equiv\Lambda^{grad}_t(X,Y)$ for any $t\in[0,T]$, by the uniqueness of the solution of the BSVIE and by Proposition \ref{BSVIEsub}.

\subsubsection{Aumann--Shapley capital allocation}

We study now the Aumann-Shapley capital allocation in the present framework. We recall that the relevance and popularity of the Aumann-Shapley CAR is mainly due to its relation with the Aumann-Shapley value and to the connection with cooperative game theory. See \cite{centrone,delnaut,tsanakas} and the references therein for details.

Consider a dynamic cash-subadditive risk measure $(\rho_t)_{t\in[0,T]}$ driven by the BSDE \eqref{CSDRM}. We can define $(\omega\times\omega)$ the following functions:
\begin{eqnarray}
&&\Lambda^{AS}_t(X,Y) \triangleq \int_0^1\mathbb{E}_{\mathbb{Q}_t^{\mu^{\gamma Y}}}\left[\left. -D^{\beta^{\gamma Y}}_tX \right|\mathcal{F}_t \right]d\gamma
    \label{ASCAR} \\
     &&\Lambda^{p-AS}_t(X,Y) \triangleq \int_0^1\Lambda_t^{sub}(X,\gamma Y)d\gamma.
     \label{penCAR}
\end{eqnarray}
Note that the two formulations above coincide when the penalty function $c_t$ is null for any $t\in[0,T]$ (e.g. for sublinear risk measures).
 We refer to $\Lambda_t^{AS}$ as Aumann-Shapley capital allocation, while to $\Lambda^{p-AS}_t$ as \textit{penalized} Aumann-Shapley CAR. These capital allocation rules have already been investigated for both the static and the dynamic case when the underlying risk measure satisfies cash-additivity (see, for example, \cite{centrone,KO,MastroRos}).

\begin{theorem}
Let $\rho_t:L^{\infty}(\mathcal{F}_T)\to L^{\infty}(\mathcal{F}_t)$ be a dynamic cash-subadditive risk measure induced by the BSDE \eqref{CSDRM} where $g$ is convex in $(y,z)$, decreasing in $y$ and satisfies the standard assumptions. Then $\Lambda_t^{AS}:L^{\infty}(\mathcal{F}_T)\times L^{\infty}(\mathcal{F}_T)\to L^{\infty}(\mathcal{F}_t)$ is an $1$-cash-subadditive, monotone, normalized, full allocating CAR, while $\Lambda_t^{p-AS}:L^{\infty}(\mathcal{F}_T)\times L^{\infty}(\mathcal{F}_T)\to L^{\infty}(\mathcal{F}_t)$
is an \textit{audacious} CAR satisfying no-undercut.

Furthermore, for any $X,Y \in L^{\infty}(\mathcal{F}_T)$,
\begin{equation*}
    \Lambda_t^{AS}(X,Y)=\mathbb{E}_{\mathbb{P}}\left[-\hat{L}^Y(T,t)X\Big|\mathcal{F}_t\right]
    \label{ASrep}
\end{equation*}
with
\begin{equation*}
    \hat{L}^Y(T,t)\triangleq\int_0^1L^{\gamma Y}(T,t)d\gamma=\int_0^1\exp\left(-\frac{1}{2}\int_t^T |\mu_s^{\gamma Y}|^2ds-\int_t^T\beta_s^{\gamma Y} ds -\int_t^T  \mu_s^{\gamma Y} dB_s\right)d\gamma,
    \label{densityAS}
\end{equation*}
where, for any fixed $\gamma\in\mathbb{R}$ and $Y\in L^{\infty}(\mathcal{F}_T)$, $(\beta_t^{\gamma Y},\mu_t^{\gamma Y})$ is an optimal scenario in the representation \eqref{BSDEmax} of $\rho_t(\gamma Y)$.
\end{theorem}

\begin{proof}
The guideline of this proof is a straightforward adaptation of Corollary 4.1 in \cite{KO} or Proposition 10 in \cite{MastroRos}, extending those results to the cash-subadditive case. Let us define $F:\mathbb{R}\to L^{\infty}(\mathcal{F}_t)$ as $F(\gamma)=\rho_t(\gamma X)$ for any fixed $t\in[0,T]$ and $X\in L^{\infty}(\mathcal{F}_T)$.  We know that if $(\beta^{H}_t,\mu^{H}_t)$ is an optimal scenario then the sub-probability $\lambda_t^{H}=D_t^{\beta^{H}}\mathbb{Q}_t^{H}$ is an element of $\partial\rho_t(H)$ for any $H\in L^{\infty}(\mathcal{F}_T)$. It holds that
\begin{align*}
   &F'_{-}(\gamma)=\lim_{h\uparrow 0}\frac{\rho_t(\gamma X+hX)-\rho_t(\gamma X)}{h}\leq \mathbb{E}_{\lambda_t^{\gamma X}}[-X|\mathcal{F}_t], \\
   &F'_{+} (\gamma)=\lim_{h\downarrow 0}\frac{\rho_t(\gamma X+hX)-\rho_t(\gamma X)}{h}\geq \mathbb{E}_{\lambda_t^{\gamma X}}[-X|\mathcal{F}_t],
\end{align*}
hence
$$\int_0^1 F'_{-}(\gamma)d\gamma\leq \int_0^1 \mathbb{E}_{\lambda_t^{\gamma X}}[-X|\mathcal{F}_t]d\gamma=\Lambda^{AS}_t(X,X) \leq \int_0^1F'_{+}(\gamma)d\gamma.$$
 By Proposition A.4 of \cite{foellmer}, it holds $F'_{-}=F'_{+}$ a.s. and the chain of equalities:
$$\int_0^1 F'_{-}(\gamma)d\gamma=\Lambda_t^{AS}(X,X) = \int_0^1 F'_{+}(\gamma)d\gamma=F(1)-F(0)=\rho_t(X).$$
Consequently, $\Lambda_t^{AS}$ is a CAR.
From Remark 7.6 in \cite{ElRa}, we know that
$$\frac{d\lambda_t^{\gamma Y}}{d\mathbb{P}}=\exp\left(-\frac{1}{2}\int_t^T |\mu_s^{\gamma Y}|^2ds-\int_t^T\beta_s^{\gamma Y} ds -\int_t^T \mu_s^{\gamma Y} dB_s \right)\triangleq L^{\gamma Y}(T,t)$$ then, by Fubini Theorem,
$$
\Lambda^{AS}_t(X,Y)=\int_0^1 \mathbb{E}_{\lambda^{\gamma Y}}[-X|\mathcal{F}_t]d\gamma =\int_0^1 \mathbb{E}_{\mathbb{P}}[-X L^{\gamma Y}(T,t)|\mathcal{F}_t]d\gamma =\mathbb{E}_{\mathbb{P}}\left[-X\int_0^1 L^{\gamma Y}(T,t)d\gamma\Big|\mathcal{F}_t\right],
$$
hence the thesis.

Monotonicity, normalization, and full allocation are due to standard properties of conditional expectation, while $1$-cash-subadditivity of $\Lambda^{AS}$ can be checked directly. For any $X,Y\in L^{\infty}(\mathcal{F}_T)$ and $m_t\in L^{\infty}(\mathcal{F}_t)$, indeed,
\begin{eqnarray*}
    \Lambda^{AS}_t(X+m_t,Y)&=&\mathbb{E}_{\mathbb{P}}\left[-(X+m_t)\int_0^1L^{\gamma Y}(T,t)d\gamma\Big|\mathcal{F}_t\right] \\
    &=&\mathbb{E}_{\mathbb{P}}\left[-X\int_0^1L^{\gamma Y}(T,t)d\gamma\Big|\mathcal{F}_t\right]+\mathbb{E}_{\mathbb{P}}\left[-m_t\int_0^1L^{\gamma Y}(T,t)d\gamma\Big|\mathcal{F}_t\right] \\
    &=&\Lambda_t^{AS}(X,Y)-m_t\int_0^1\mathbb{E}_{\mathbb{Q}_t^{\mu^{\gamma Y}}}\left[\exp\left(-\int_t^T\beta_s^{\gamma Y}ds\right)\Bigg|\mathcal{F}_t\right]d\gamma \\
    &\geq & \Lambda_t^{AS}(X,Y)-m_t,
\end{eqnarray*}
where the last inequality follows from $\beta^{\gamma Y}_t\geq 0$ for any $\gamma \in[0,1]$.

It remains to prove the part concerning $\Lambda^{p-AS}_t(X,Y)$. It holds that
\begin{align}
    \Lambda^{p-AS}_t(X,X)=\Lambda_t^{AS}(X,X)-\int_0^1c_t(D_t^{\beta^{\gamma X}}\mathbb{Q}_t^{\mu^{\gamma X}})d\gamma=\rho_t(X)-\int_0^1c_t(D^{\beta^{\gamma X}}_t\mathbb{Q}^{\mu^{\gamma X}}_t)d\gamma,
\end{align}
where the penalty term $c_t(D^{\beta^{\gamma X}}_t\mathbb{Q}^{\mu^{\gamma X}}_t)$ is given in \eqref{minpen}.
Then
$\Lambda^{AS}_t(X,Y)\leq \rho_t(X)$, i.e. it is an audacious CAR. Furthermore, the no-undercut property follows from
$$\Lambda^{p-AS}_t(X,Y)=\int_0^1 \left( \mathbb{E}_{\lambda_t^{\gamma Y}}[-X|\mathcal{F}_t]-c_t(\lambda_t^{\gamma Y}) \right) \, d\gamma \leq \int_0^1\rho_t(X)d\gamma=\rho_t(X).$$
\end{proof}

\section{Further examples}\label{sec: examples}

In this section, we provide some examples of CARs driven by BSVIEs with different drivers. More precisely, we show that the marginal capital allocation and a generalized marginal capital allocation can be obtained via the approach in \eqref{CARBSVIE}. A further example will deal with the here defined \textit{cash-subadditive entropic risk measure} (CSERM) that allows to ``build'' different CARs reflecting investor's beliefs.

\subsection{Dynamic marginal capital allocation rule}

Let $\rho_t$ be a dynamic cash-subadditive risk measure driven by a BSDE.
Consider the dynamic marginal capital allocation rule, defined as
\begin{equation*}
\Lambda^{M}_t(X,Y) \triangleq\rho_t(Y)-\rho_t(Y-X), \quad X,Y \in L^{\infty}(\mathcal{F}_T).
\end{equation*}
Financially speaking, such an allocation quantifies the contribution of sub-portfolio $X$ to the riskiness of the overall position $Y$. See Tasche \cite{Tasche} in the static case for details.

The dynamic marginal capital allocation can be obtained through the construction in \eqref{CARBSVIE} by means of a BSVIE. Indeed,
    \begin{eqnarray*}
       \hspace{-0.8cm}&& \Lambda^{M}_t(X,Y) =\Lambda^{sub}_t(Y,Y)-\Lambda^{sub}_t(Y-X,Y-X) \\
        \hspace{-0.8cm}&& = -D_tX \\
        \hspace{-0.8cm}&&+\hspace{-0.1cm}\int_t^T \hspace{-0.1cm}\left[g_{\Lambda^s}(t,s,\rho_s(Y),Z^{\rho(Y)}_s,Z^{\Lambda^s(Y,Y)}(t,s))
         \hspace{-0.1cm} - \hspace{-0.1cm}g_{\Lambda^{s}}(t,s,\rho_s(Y-X),Z_s^{\rho(Y-X)},Z^{\Lambda^s(Y-X,Y-X)}(t,s)) \right] \hspace{-0.1cm}ds \\
        \hspace{-0.6cm}&& -\int_t^T \left( Z^{\Lambda^s(Y,Y)}(t,s)-Z^{\Lambda^s(Y-X,Y-X)}(t,s) \right) \, dB_s.
        \end{eqnarray*}
       Defining $Z^{\Lambda^M}(t,s)\triangleq Z^{\Lambda^s(Y,Y)}(t,s)-Z^{\Lambda^s(Y-X,Y-X)}(t,s)$ and
        \begin{eqnarray*}
        \hspace{-0.4cm}&& g_{\Lambda^M}(t,s,\rho_s(Y),Z^{\rho(Y)}_s,Z^{\Lambda^s(Y)}(t,s),Z^{\Lambda^M}(t,s))
         \\
        \hspace{-0.4cm} && \triangleq  g_{\Lambda^s}(t,s,\rho_s(Y),Z^{\rho(Y)}_s,Z^{\Lambda^s(Y,Y)}(t,s))
        -g_{\Lambda^s}(t,s,\rho_s(Y-X),Z_s^{\rho(Y-X)},Z^{\Lambda^{s}(Y,Y)}(t,s)-Z^{\Lambda^M_s}(t,s)),
         \end{eqnarray*}
         then $g_{\Lambda^M}$ satisfies condition \eqref{HPg} when $g_{\Lambda^s}(t,s,\rho,Z^{\rho},0)=0$ for any $(t,s)\in\Delta$, $(\rho,Z^{\rho})\in\mathbb{R}\times\mathbb{R}^n$ (true, for instance, when $\rho_t$ is positively homogeneous).
         It then follows that the approach in \eqref{CARBSVIE} covers also the case of the dynamic marginal CAR.
         
        \subsection{Generalized marginal dynamic CARs}

        Let us consider again a cash-subadditive risk measure $\rho_t$ induced via a BSDE. Starting from the classical marginal capital allocation rule seen before, we now define a dynamic generalized marginal capital allocation rule.

        Given a stochastic adapted $\mathbb{R}^n$-valued process $(\lambda_t)_{t\in[0,T]}\in \mbox{BMO}(\mathbb{P})$ (see Section 7.1.2 \cite{BaEl} for further details), we now consider a dynamic CAR $\Lambda$ with the following driver
         \begin{eqnarray*}
             && g_{\Lambda}(t,s,\rho_s(Y),Z^{\rho(Y)}_s,Z^{\Lambda^s(Y,Y)}(t,s),Z^{\Lambda}(t,s))\\
             &&= g_{\Lambda^s}(t,s,\rho_s(Y),Z^{\rho(Y)}_s,Z^{\Lambda^s(Y,Y)}(t,s))
             + \lambda_s (Z^{\Lambda}(t,s)-Z^{\Lambda^s(Y,Y)}(t,s)).
         \end{eqnarray*}
         In such a case, condition \eqref{HPg} is automatically satisfied.

         As shown in a while, the corresponding $\Lambda$ can be interpreted as a dynamic generalized marginal capital allocation since
         \begin{equation*}
            \Lambda_t(X,Y)=\rho_t(Y)-\rho^{\lambda}_t(Y-X),
        \end{equation*}
         where $\rho^{\lambda}$ is a dynamic risk measure that will defined here below and depending on $(\lambda_t)_{t\in[0,T]}$.

         Evaluating the difference between $\Lambda_t(X,Y)$ and $\rho_t(Y)$ we get:
         \begin{eqnarray*}
             && \Lambda_t(X,Y)-\rho_t(Y)=\Lambda_t(X,Y)-\Lambda_t^{sub}(Y,Y) \\
             &&= -D_t(X-Y)+\int_t^T \lambda_s (Z^{\Lambda}(t,s)-Z^{\Lambda^s(Y,Y)}(t,s)) \, ds -\int_t^T \left(Z^{\Lambda}(t,s)-Z^{\Lambda^s(Y,Y)}(t,s)\right) \, dB_s  \\
             &&= -D_t(X-Y)-\int_t^T \left( Z^{\Lambda}(t,s)-Z^{\Lambda^s}(t,s) \right) \, dB^{\lambda}_s,
         \end{eqnarray*}
         where $dB^{\lambda}_t\triangleq
         dB_t-\lambda_tdt$ is a Brownian motion w.r.t. $(\mathcal{F}_t)_{t\in[0,T]}$ because $(\lambda_t)_{t\in[0,T]}\in \mbox{BMO}(\mathbb{P})$, while $D_t=D_t^Y= \exp\{-\int_t^T \beta_s^Y ds \}$ where $\beta^Y_t$ is the first component of the optimal scenario in the representation of $\rho_t(Y)$. Taking the conditional expectation in the previous expression, it holds that
         \begin{equation*}
             \Lambda_t(X,Y)-\rho_t(Y)=\mathbb{E}_{\mathbb{Q}^{\lambda}}[-D_t(Y-X)\vert \mathcal{F}_t],
         \end{equation*}
        where $\frac{d\mathbb{Q}^{\lambda}}{d\mathbb{P}}\triangleq\mathcal{E}(\int_0^{\cdot}\lambda_t dB_t)_T$ is the measure corresponding to $B^{\lambda}$, defined via the stochastic exponential. $\mathbb{E}_{\mathbb{Q}^{\lambda}}[-D_t(Y-X)\vert \mathcal{F}_t]$ turns out to be the component $\rho^{\lambda}_t(Y-X)$ of the (unique) solution to the linear BSDE:
        \begin{eqnarray*}
            \rho^{\lambda}_t(Y-X) & =& -(Y-X)+\int_t^T \left(-\beta^Y_s\rho_s^{\lambda}+ \lambda_s Z^{\rho^{\lambda}}_s\right) ds-\int_t^TZ_s^{\rho^{\lambda}} \,dB_s \\
            & =& -(Y-X)+\int_t^T-\beta^Y_s\rho_s^{\lambda}ds-\int_t^T Z^{\rho^{\lambda}}_s \, dB^{\lambda}_s.
        \end{eqnarray*}
        $\rho^{\lambda}_t$ is then a cash-subadditive risk measure because the driver $g_{\rho^{\lambda}}(t,\rho^{\lambda},Z^{\rho^{\lambda}})=-\beta^Y_t\rho^{\lambda}+ \lambda_t Z^{\rho^{\lambda}}$  is convex (or, better, linear) in $(\rho^{\lambda},Z^{\rho^{\lambda}})$ and non-increasing in $\rho^{\lambda}$. It then follows that
        \begin{equation*}
            \Lambda_t(X,Y)=\rho_t(Y)-\rho^{\lambda}_t(Y-X).
        \end{equation*}
        Such a $\Lambda$ can be therefore interpreted as a generalized marginal CAR,
         being the difference between the risk associated to the total portfolio $Y$ through the underlying risk measure $\rho_t$ and the marginal contribution of the sub-portfolio $(Y-X)$ evaluated via a different cash-subadditive risk measure $\rho_t^{\lambda}$. Furthermore, $\rho_t^{\lambda}$ can be understood as the degree of confidence- modulated by $\lambda_s$- of the investor in the market num\'{e}raire (see Section 6.2 in Barrieu and El Karoui \cite{BaEl}) when the agent also faces with an ambiguous interest rate $\beta^Y_t$.

        \subsection{Cash-subadditive entropic risk measures and capital allocations}

       In this section, we extend the classic entropic risk measure (solving a BSDE with driver $g_{\rho^E}(t,Z^{\rho^E})=\frac{|Z^{\rho^E}|^2}{2\gamma}$ where $\gamma>0$ is the risk aversion of the agent) to the cash-subadditive setting.
       Inspired by the work \cite{ElRa},  we consider a driver whose local behaviour is:
       $$\mathbb{E}_{\mathbb{P}}[-d\rho^{SE}_t\vert\mathcal{F}_t]=g_{\rho^{SE}} \, dt\triangleq \left(-\beta_t\rho^{SE}+\frac{|Z^{\rho^{SE}}|^2}{2\gamma} \right)dt,$$
        where $\beta_t$ is a positive, bounded and adapted process with the usual interpretation as an ambiguous discount rate. In particular, this risk measure is locally compatible with an agent with a certain risk aversion coefficient $\gamma>0$ and adverse to interest rate ambiguity (see \cite{BaEl,ElRa} for a thorough discussion). Clearly, when no ambiguity is possible, $\beta_t\equiv 0$ and we regain the classic entropic risk measure. It is evident that $g_{\rho^{SE}}$ does not satisfy a Lipschitz condition, being quadratic in $Z^{\rho^{SE}}$. Nevertheless, the theory we have developed is still robust. Indeed, representation \eqref{BSDEmax} holds also for risk measures with quadratic growth (see Theorem 7.5 in \cite{ElRa}). Moreover, $g_{\rho^{SE}}$ is differentiable both in $\rho^{SE}$ and $Z^{\rho^{SE}}$, hence
        \begin{eqnarray*}
        &&g_{\Lambda^s}(t,s,\rho^{SE},Z^{\rho^{SE}},Z^{\Lambda^s}) \\
        &&=D_{t,s}\left[-g_{\rho^{SE}}(s,\rho^{SE}_s,Z^{\rho^{SE}}_s)-\beta_s\rho_s^{SE}+ \frac{|Z^{\rho^{SE}}|^2}{\gamma} \right]-\frac{1}{\gamma}  Z^{\rho^{SE}}_s Z^{\Lambda^s}(t,s) \\
        &&=\frac{D_{t,s}}{2\gamma}|Z^{\rho^{SE}}_s|^2-\frac{1}{\gamma} Z^{\rho^{SE}}_s Z^{\Lambda^s}(t,s),
        \end{eqnarray*}
        because of the relation between gradient and subdifferential:
        \begin{eqnarray*}
        &&\partial^{sub}_{\rho}g(t,\rho^{SE}_t,Z^{\rho^{SE}}_t)=-\partial_{\rho}g(t,\rho_t^{SE},Z^{\rho^{SE}}_t)=\beta_t, \\
        &&\partial^{sub}_{z}g(t,\rho^{SE}_t,Z^{\rho^{SE}}_t)=-\partial_{z}g(t,\rho_t^{SE},Z^{\rho^{SE}}_t)=-\frac{Z^{\rho^{SE}}_t}{\gamma}.
        \end{eqnarray*}
        Although it is not clear if the BSVIE with driver $g_{\Lambda^s}(t,s,\rho^{SE},Z^{\rho^{SE}},Z^{\Lambda^s})$ admits a unique solution (since the Lipschitz coefficient $Z^{\rho^{SE}}_s$ is no longer bounded and $g(t,s,\rho^{SE},Z^{\rho^{SE}},0)=\frac{D_{t,s}}{2\gamma}|Z^{\rho^{SE}}_s|^2$ is not square integrable), the existence and uniqueness of the solution is proved in Lemma \ref{LemmaEU} here below.

        We define now a generalized notion of entopic risk measure in the cash-subadditive case, with interest rate $(\beta_t)_{t\in[0,T]}$ (same as before) and different risk aversion parameter $\gamma_1>0$, i.e. $$g_{\rho^{\gamma_1}}(t,\rho^{\gamma_1},Z^{\rho^{\gamma_1}})=-\beta_t\rho^{\gamma_1}+\frac{1}{2\gamma_1} |Z^{\rho^{\gamma_1}}|^2.$$
        The corresponding dynamic risk measure is a cash-subadditive and convex risk measure that will be called \textit{cash-subadditive entropic risk measure} (CSERM).

        Set now
        \begin{eqnarray*}
        &&g_{\Lambda^{SE}}(t,s,\rho^{SE}_s(Y),Z^{\rho^{SE}(Y)}_s,Z^{\Lambda^s(Y,Y)}(t,s),Z^{\Lambda^{SE}}(t,s))\\
        &&\triangleq g_{\Lambda^s}(t,s,\rho^{SE}_s(Y),Z^{\rho^{SE}(Y)}_s,Z^{\Lambda^s(Y,Y)}(t,s))+
         \frac{2}{\gamma_1}Z^{\rho^{\gamma_1}}_s (Z^{\Lambda^{SE}}(t,s)-Z^{\Lambda^s}(t,s)).
        \end{eqnarray*}
        Lemma \ref{LemmaEU} guarantees then that the BSVIE with driver $g_{\Lambda^{SE}}$ admits a unique solution. As seen below, if a solution (with whatsoever regularity) is provided then $\Lambda^{SE}_t$ is a bounded process. Following the same steps as in the generalized marginal CARs case,
       \begin{equation*}
           \Lambda^{SE}_t(X,Y)=\rho^{SE}_t(Y)-\rho_t^{\nabla}(Y-X),
       \end{equation*}
        where
        $$\rho_t^{\nabla}(Y-X)=-(Y-X)+\int_t^T \left[-\beta_s\rho_s^{\nabla}(Y-X)+\frac{2}{\gamma_1}  Z_s^{\rho^{\gamma_1}} Z_s^{\rho^{\nabla}} \right] ds-\int_t^TZ^{\rho^{\nabla}}_sdB_s$$
        admits a unique bounded solution. Furthermore,
        \begin{equation*}
        \Lambda^{SE}_t(X,Y)=\rho_t^{SE}(Y)-\nabla\rho_t^{\gamma_1}(Y-X;Y),
        \label{SEcar}
        \end{equation*}
    where $\nabla\rho_t^{\gamma_1}$ is the gradient capital allocation considering $\rho_t^{\gamma_1}$ as underlying risk measure (see Kromer and Overbeck \cite{KO} for further details).
        Indeed,
        \begin{equation}
        	\rho^{\nabla}_t(Y-X)=\nabla\rho^{\gamma_1}_t(Y-X;Y) \mbox{ a.s.}
        \label{gradientP}
        \end{equation}
        because, according to Theorem 2.1 in \cite{Ankirchner},
        $$\nabla\rho^{\gamma_1}_t(Y-X;Y)=-(Y-X)+\int_t^T \left[-\beta_s\nabla\rho^{\gamma_1}_s(Y-X;Y)+\frac{2}{\gamma_1}  Z^{\rho^{\gamma_1}}_s \nabla Z^{\rho^{\gamma_1}}_s \right]ds-\int_t^T\nabla Z^{\rho^{\gamma_1}}_s dB_s$$
        and by uniqueness of the solution.

        Summing up, the capital allocation rules associated to the driver $g_{\Lambda^{SE}}$ can be expressed in terms of the difference between the risk of the total portfolio $Y$ measured by the cash-subadditive entropic risk measure with risk aversion $\gamma$ and the risk of the margin portfolio $(Y-X)$ given by the Gateaux derivative of the CSERM  with risk aversion parameter $\gamma_1$.\medskip
        
Finally, we prove the result on the existence and uniqueness in a weaker sense mentioned before.
\begin{lemma}\label{LemmaEU}
	In the previous setting, there exists a unique solution $(\Lambda^{sub},Z^{s})\in \mathcal{H}^{1,\infty}$ to the BSVIE
	\begin{equation}
		\Lambda^{sub}_t(X,Y)=-D_tX+\int_t^T \left[ \frac{1}{2\gamma}D_{t,s}|Z^{\rho^{SE}}_s|^2-\frac{1}{\gamma}  Z^{\rho^{SE}}_s Z^{\Lambda^s}(t,s) \right]ds-\int_t^T Z^{\Lambda^s}(t,s)dB_s.
\label{eqnr}
	\end{equation}
Moreover, the BSVIE
	\begin{equation}
		\label{eqCSERM}
		\Lambda_t^{SE}(X,Y)=-D_tX+\int_t^T g_{\Lambda^{SE}}(t,s,\rho_s^{SE}(Y),Z^{\rho^{SE}}_s,Z^{\Lambda^s(Y,Y)}(t,s),Z^{SE}(t,s))ds-\int_t^T Z^{SE}(t,s)dB_s
	\end{equation}
	admits a unique solution $(\Lambda^{SE},Z^{\Lambda^{SE}})\in \mathcal{H}^{1,\infty}$ with the further regularity $\Lambda^{SE}\in  L^{\infty}((0,T),\mathbb{F}).$
\end{lemma}
\begin{proof}
	We start proving that there exists a unique solution to \eqref{eqnr}. We know that $Z^{\rho^{SE}}_t\in \mbox{BMO}(\mathbb{P})$ (see \cite{Wang}), then $dB_t^{\rho^{SE}}=dB_t+\frac{1}{\gamma}Z^{\rho^{SE}}_tdt$ is a Brownian motion w.r.t. $\mathbb{Q}^{\rho^{SE}}=\mathcal{E}\left(\int_0^{\cdot}\frac{1}{\gamma}Z^{\rho^{SE}}_tdB_t\right)_T$. Hence, equation \eqref{eqnr} becomes
	$$\Lambda^{sub}_t(X,Y)=-D_tX+\int_t^T\frac{1}{2\gamma}D_{t,s}|Z^{\rho^{SE}}_s|^2 ds-\int_t^T Z^{\Lambda^s}(t,s)dB^{\rho^{SE}}_s=\varphi(t)-\int_t^T Z^{\Lambda^s}(t,s)dB^{\rho^{SE}}_s,$$
with $$\varphi(t)=-D_tX+\int_t^T\frac{1}{2\gamma}D_{t,s}|Z^{\rho^{SE}}_s|^2 ds.$$ We observe that $\varphi\in L^{1}(\Omega,\mathcal{F}_T;L^{\infty}(0,T))$, given that $D_t\in L^{\infty}((0,T),\mathcal{F}_T)$ and $Z^{\rho^{SE}}\in\mbox{BMO}(\mathbb{P})$. The last equation enables us to avoid dealing with stochastic Lipschitz coefficients. 
	Let us consider the family of BSDEs whose elements are defined for each fixed $(t,s)\in\Delta$ as follows:
	$$\eta(s;t)=\varphi(t)-\int_s^T\zeta(r;t)dB_r^{\rho^{SE}}, \ \ (t,s)\in\Delta.$$
There exists a unique solution $(\eta(\cdot;t),\zeta(\cdot;t))\in L^1(\Omega,\mathbb{F},L^{\infty}(t,T))\times L^1(\Omega,\mathbb{F};L^2(t,T))$ for each $t\in[0,T]$, by Theorem 3 in \cite{fan}. Moreover, we have:
 \begin{align*}
|\eta(s;t)|&\leq\mathbb{E}\left[|\varphi(t)|\big|\mathcal{F}_s\right]\leq C\mathbb{E}\left[\left.|X|+\int_t^T|Z^{\rho^{SE}}_r|^2 dr\right|\mathcal{F}_s\right] \\
&\leq C\left(\|X\|_{\infty}+\mathbb{E}\left[\int_t^s|Z^{\rho^{SE}}_r|^2 dr\right]+\mathbb{E}\left[\left. \int_s^T|Z^{\rho^{SE}}_r|^2ds\right|\mathcal{F}_s\right]\right) \\
&\leq C\left(\|X\|_{\infty}+\mathbb{E}\left[\int_t^T|Z^{\rho^{SE}}_r|^2dr\right]+\sup_{m\in[t,T]}\mathbb{E}\left[\left. \int_m^T|Z^{\rho^{SE}}_r|^2dr\right|\mathcal{F}_m\right]\right)\leq C,
  \end{align*}
where $C>0$ is a constant that does not depend on $(t,s)\in\Delta$ and can vary from line to line. Second inequality follows from the boundedness of $D_t$ while last inequality is due to $Z^{\rho^{SE}}\in\mbox{BMO}(\mathbb{P})$. Thus, we have:
$$\mathbb{E}\left[\sup_{(t,s)\in\Delta}|\eta(s;t)|\right]\leq C <+\infty.$$
By Proposition 1 in \cite{fan}, for any $(t,s)\in\Delta$ we have the estimate:
$$\mathbb{E}\left[\left(\int_s^T|\zeta(r;t)|^2dr\right)^{1/2}\right]\leq C\mathbb{E}\left[\sup_{m\in[s,T]}|\eta(m;t)|\right],$$ taking the sup over $t\in[0,T]$ we have:
$$\sup_{t\in[0,T]}\mathbb{E}\left[\left(\int_s^T|\zeta(r;t)|^2dr\right)^{1/2}\right]\leq C\sup_{t\in[0,T]}\mathbb{E}\left[\sup_{m\in[s,T]}|\eta(m;t)|\right]\leq C'<+\infty.$$
Setting $\Lambda^{sub}(t)\equiv \eta(t;t)$ and $Z^{\Lambda^s}(t,s)\equiv \zeta(s;t)$ with $(t,s)\in\Delta$, the thesis follows (uniqueness is given by the previous estimates together with linearity of the equation, by taking $\varphi\equiv 0$).

Now we are ready to show the existence and uniqueness for the BSVIE \eqref{eqCSERM}.
	We are going to build the solution $(\Lambda^{SE}_t,Z^{SE}(t,s))$ by means of the solution to the BSVIE:
	$$\Lambda'_t(X-Y)=-D_t(X-Y)-\int_t^T Z^{\Lambda'}(t,s)dB^{\gamma_1}_s,$$ where $dB^{\gamma_1}_s\triangleq dB_s-\frac{2}{\gamma_1}Z^{\gamma_1}_sds$ is a $\mathbb{Q}^{\gamma_1}$-Brownian motion given that $\frac{2}{\gamma_1}Z^{\gamma_1}_s\in \mbox{ BMO}(\mathbb{P})$. The last equation admits a unique bounded solution according to Theorem \ref{LinfBSVIE}. Set now
	\begin{eqnarray*}
		&&\Lambda^{SE}_t(X,Y)\triangleq\Lambda'_t(X-Y)+\rho_t^{SE}(Y)=\Lambda'_t(X-Y)+\Lambda^{sub}_t(Y,Y)\in L^{\infty}([0,T],(\mathcal{F}_t)_{t\in[0,T]};\mathbb{R}) \\
		&&Z^{SE}(t,s)\triangleq Z^{\Lambda'}(t,s)+Z^{\Lambda^s(Y,Y)}(t,s)\in L^{\infty}((0,T);L^1(\Omega,(\mathcal{F}_t)_{t\in[t,T]};L^{2}(t,T))).
	\end{eqnarray*}
	By construction, $(\Lambda^{SE}_t(X,Y),Z^{\Lambda^{SE}}(t,s))$ is a solution of equation \eqref{eqCSERM} with the above mentioned regularity.
	It remains to prove uniqueness of the solution. Let us consider another solution $(\Lambda^1_t,Z^{\Lambda^1}(t,s))\in \mathcal{H}^{1,\infty}$ to equation \eqref{eqCSERM}. Then
	\begin{eqnarray*}
		&&\Lambda^{SE}_t(X,Y)-\Lambda^1_t(X,Y) \\
		&=&\int_t^T\frac{2}{\gamma_1}  Z^{\rho^{\gamma_1}}_t (Z^{\Lambda^{SE}}(t,s)-Z^{\Lambda^1}(t,s)) \, ds-\int_t^T (Z^{\Lambda^{SE}}(t,s)-Z^{\Lambda^1}(t,s)) \, dB_s \\
		&=&\int_t^T (Z^{\Lambda^{SE}}(t,s)-Z^{\Lambda^1}(t,s)) \, dB^{\gamma_1}_s,
	\end{eqnarray*}
where $dB_t^{\gamma_1}\triangleq dB_t-\frac{2}{\gamma_1}Z^{\rho^{\gamma_1}}_tdt$ is a $\mathbb{Q}^{\gamma_1}$-Brownian motion. Since the last BSVIE admits a unique solution given by $(\Lambda^{SE}_t-\Lambda^1_t\equiv 0,Z^{\Lambda^{SE}}(t,s)-Z^{\Lambda^1}(t,s)\equiv 0)\in\mathcal{H}^{1,\infty}$ (thanks to the previous results), the uniqueness then follows.
	\label{EUrem}
\end{proof}

\section{Conclusions} \label{sec: conclus}

To sum up, in this paper we have focused on capital allocations associated to cash-subadditive risk measures, adopting an axiomatic approach without assuming Gateaux differentiability and providing examples where the underlying risk measure is driven by a BSDE. To the best of our knowledge, this is the first attempt to deal with dynamic capital allocation rules when the underneath risk measure does not fulfil the axiom of cash-additivity.
More precisely, we have extended the results obtained in Mastrogiacomo and Rosazza Gianin \cite{MastroRos1} and in Kromer and Overbeck \cite{KO} to the cash-subadditive case. At the same time, differently from Centrone and Rosazza Gianin \cite{centrone} dealing with static CARs associated to convex and quasi-convex risk measures, we encompass in our treatment also the dynamic setting.

For risk measures induced by BSDEs, \cite{MastroRos1,KO} showed that the corresponding capital allocations follow still a BSDE; in particular, this result is due to the axiom of cash-additivity. On the Volterra side, instead, Kromer and Overbeck \cite{KO1} proved that the dynamic gradient capital allocation of risk measures generated by a BSVIE is still given by a BSVIE.
Differently from these approaches mapping either from BSDEs to BSDEs or from BSVIEs to BSVIEs in the cash-additive case, here we have provided an alternative way to build capital allocation rules starting from cash-subadditive risk measures induced by BSDEs.

Compared to the existing literature, the main novelties of our work consist of: firstly, weakening cash-additivity with cash-subadditivity at the level of risk measures and evaluating the impact on capital allocations; secondly, proving that the subdifferential CAR associated to a cash-subadditive risk measure induced by a BSDE follows a BSVIE; thirdly, having introduced a general way to build different CARs given by BSVIEs. This general approach allows to treat the cash-subadditive case similarly as done for cash-additive risk measures in \cite{MastroRos1}, switching from probabilities to sub-probabilities and from BSDEs to BSVIEs.
\bigskip
\bigskip

\textbf{Acknowledgements:} The authors are members of Gruppo Nazionale per l'Analisi Matematica, la Probabilit\`{a} e le loro Applicazioni (GNAMPA), Italy. The authors thank Ilaria Peri and the participants at the Quantitative Finance Workshop 2023 in Gaeta, Italy, for comments and discussions.

\end{document}